\documentclass[11pt,makeidx]{amsart}

\usepackage{amsthm,amsfonts,amssymb,amsmath,amscd,amsrefs,thmtools}
\usepackage[all]{xy}
\usepackage{xr-hyper}
\usepackage{color}
\definecolor{crimsonglory}{rgb}{0.75, 0.0, 0.2}
\definecolor{darkblue}{rgb}{0.0, 0.0, 0.55}
	\definecolor{deepskyblue}{rgb}{0.0, 0.75, 1.0}

\usepackage[colorlinks=true,linkcolor = blue, citecolor = deepskyblue]{hyperref}

\usepackage[OT2,OT1]{fontenc}
\usepackage{mathrsfs}
\usepackage{multirow}
\usepackage{graphicx}

\newtheoremstyle{mystyle}
  {8pt}
  {3pt}
  {\em}
  {}
  {\scshape}
  {.}
  {3pt}
  {}

\newcommand{\tr}{\text{tr}}
\newcommand{\SL}{\text{SL}_2}
\newcommand{\EL}{\text{EL}_2}
\newcommand{\PSL}{\text{PSL}_2}
\newcommand{\M}{\mathbf{M}}
\newcommand{\Fp}{\mathbb{F}_p}
\newcommand{\Fpp}{\mathbb{F}_{p^2}}
\newcommand{\Tr}{\text{Tr}}

\newcommand{\Z}{\mathbb{Z}}
\newcommand{\GL}{\text{GL}_2}
\newcommand{\E}{\text{E}}
\newcommand{\Orb}{\mathcal{O}}
\newcommand{\e}{\text{e}}
\newcommand{\cc}{\mathcal{C}}
\newcommand{\ee}{\mathcal{E}}

\theoremstyle{mystyle}

\newcounter{exe}
  
\newtheorem{thm}{\textsc{Theorem}}

\newtheorem{cor}[exe]{\textsc{Corollary}}
\newtheorem{conj}{\textsc{Conjecture}}
\newtheorem{lemma}[exe]{\textsc{Lemma}}
\newtheorem{prop}[exe]{\textsc{Proposition}}

\theoremstyle{definition}

\usepackage{color}

\title[]{Markoff Triples and Generating Pairs of $\SL(\Fp)$}
\author{Jo\~ao Campos-Vargas}
\email{joaoccv@stanford.edu}
\address{Department of Mathematics, Princeton University, Princeton, NJ 08544, USA.}

\begin{document}

\begin{abstract}

Consider the level sets of the Markoff equation
\[\M_k: x^2 + y^2 + z^2 - xyz - 2 = k.\]
The phenomenon of \textit{strong approximation}, as named by Bourgain, Gamburd, and Sarnak, predicts that every solution of $\M_k$ over $\Fp$ descends from a solution over $\Z$. Moreover, we expect that the action of Vieta involutions (taking $(x, y, z)$ to $(yz-x, y, z)$, $(x, xz-y, z)$, and $(x, y, xy-z)$) on $\M_k(\Fp)$ is essentially transitive. In terms of matrices, Vieta involutions correspond to Nielsen moves in pairs $(A, B) \in \SL(\Fp) \times \SL(\Fp)$ for which $\tr([A, B]) = k$. This correspondence is induced by
\[\Tr: (A, B) \mapsto (\tr(A), \tr(B), \tr(AB)).\]

\par McCullough and Wanderley conjectured that Nielsen moves connect two pairs $(A_1, B_1)$, $(A_2, B_2)$ of generators of $\SL(\Fp)$ if and only if $[A_1, B_1]$ is conjugate to $[A_2, B_2]$ or $[B_2, A_2]$. Based on this, one expects that generating pairs $(A, B)$ of $\SL(\Fp)$ for which $\tr([A, B]) = k$ determine a single orbit of $\M_k(\Fp)$, and the remaining \textit{exceptional} orbits come from non-generating pairs of $\SL(\Fp)$.

\par In this article, we describe the set of exceptional orbits of $\M_k(\Fp)$, showing that they agree with the finite orbits of the equation $\M_k$ over $\mathbb{C}$ found by Dubrovin and Mazzocco. Furthermore, we prove that the conjecture of McCullough and Wanderley is equivalent to strong approximation when $p \equiv 3 \mod{4}$. Lastly, we present the recent developments of Chen on the problem and use our classification of exceptional orbits to make a divisibility conjecture about the size of the largest orbit of $\M_k(\Fp)$.
\end{abstract}

\maketitle

\section*{Introduction}\label{intro}

\par While studying the problem of Diophantine approximations, Markoff came across a set of numbers $\mathcal{M}$ corresponding to the positive coordinates of solutions of the equation
\[\mathbf{X}: x^2 + y^2 + z^2 = 3xyz.\]
The set $\mathcal{M}$ proved to be remarkably complex, and its elements -- known as Markoff numbers -- seem to be better understood when viewed dynamically. Namely, if $\mathbf{X}(\mathbb{Z})$ denotes the set of integer solutions of $\mathbf{X}$ and $(x, y, z)$ is one of its elements, then its coordinate permutations together with the Vieta involution $(3yz - x, y, z)$ are also in $\mathbf{X}(\mathbb{Z})$. This can be viewed as a group action which partitions $\mathbf{X}(\mathbb{Z})$ into different orbits. Remarkably, Markoff showed that $\mathbf{X}(\mathbb{Z}) \setminus \{(0, 0, 0)\}$ consists of a single orbit, so that the action is nearly transitive.

\par Similarly, we define $\mathbf{X}(\mathbb{F}_p)$ as the set of solutions of $\mathbf{X}$ over $\mathbb{F}_p$. Let $\pi: \mathbf{X}(\mathbb{Z}) \to \mathbf{X}(\mathbb{F}_p)$ be the map given by reduction modulo $p$. As for any other object that can be defined over $\mathbb{Z}$ and $\mathbb{F}_p$, it is natural to ask whether or not the map $\pi$ is surjective. This type of problem is known as a strong approximation problem and it is the central theme of this work. Notice that in this case, this question is equivalent to the classification of orbits of $\mathbf{X}(\mathbb{F}_p)$, and the map $\pi$ will be surjective if and only if $\mathbf{X}(\mathbb{F}_p)$ consists of exactly two orbits: $\{(0, 0, 0)\}$ and $\mathbf{X}(\Fp) \setminus \{(0, 0, 0)\}$.

\par Upon a change of variables, we notice that $\mathbf{X}$ is part of a family of expressions carrying a more subtle algebraic structure. These are given by
\[\mathbf{M}_{k}: x^2 + y^2 + z^2 - xyz - 2 = k,\]
where our original equation $\mathbf{X}$ corresponds to $\mathbf{M}_{-2}$. Although in general $\mathbf{M}_k(\mathbb{Z})$ has a more complicated orbit structures for generic values of $k$, it has been conjectured that $\mathbf{M}_k(\mathbb{F}_p)$ consists of essentially one orbit when $k \ne 2$. The main goal of this work is to give a precise formulation for this conjecture together with a general picture of the theory.

\par The algebraic structure behind the equations $\mathbf{M}_k$ comes from an identity discovered by Fricke which precedes Markoff's work. This identity relates triples of elements with pairs of elements in the two dimensional special linear group. The study of such pairs over $\mathbb{F}_p$ appears in the work of Macbeath and provides enough information to describe small orbits of $\mathbf{M}_k(\mathbb{F}_p)$. Remarkably, these orbits appear in the work of Dubrovin and Mazzocco as the only finite orbits for the equations $\M_k$ over $\mathbb{C}$. 

\par A precise picture of these small orbits will allow us to make a conjectural statement for the problem of strong approximation previously described. Moreover, it gives us an analytic picture of the Q-classification conjecture proposed by McCullough and Wanderley on generating pairs of $\text{SL}_2(\mathbb{F}_p)$, which further motivates the study of the problem.

\par The existence of a very large orbit containing most points of $\mathbf{M}_k(\mathbb{F}_p)$ comes from the remarkable work of Bourgain, Gamburd, and Sarnak. Their result shows, effectively, that there are at most $p^{\epsilon}$ points outside of the main orbit of $\mathbf{M}_k(\mathbb{F}_p)$. Through this result, the conjecture is reduced to the statement that the only points in $\mathbf{M}_k(\mathbb{F}_p)$ with orbit size bounded by $p^{\epsilon}$ are those in the small orbits derived from Macbeath's work, which in fact have absolutely bounded size.

\par The most recent developments in the problem come with the work of Chen, proving divisibility conditions for the size of the orbits of $\mathbf{M}_k(\mathbb{F}_p)$. In the case $k = -2$, Chen was able to show that $p$ divides the size of all orbits except $\{(0, 0, 0)\}$, which is enough to prove the conjecture. For generic values of $k$, the expected size of the main orbit of $\mathbf{M}_k(\mathbb{F}_p)$ still satisfies strong divisibility conditions, although this is not implied by Chen's work.

\par In our presentation, we try to follow the natural order of developments in the problem. Section \ref{foundations} lays out the foundations, introducing the Markoff equation, Fricke's identity, and Nielsen moves. In Section \ref{technical}, we present some technical results that will be used in later sections, including a classification of subgroups of $\PSL(\Fp)$. Section \ref{macbeath} covers results of Macbeath that give structure to the trace map. In Section \ref{main}, we expand on Macbeath's work and classify all exceptional orbits of $\M_k(\Fp)$. 

\par Section \ref{strong} combines the previous sections to give an analytic formulation of strong approximation and an exposition of the Q-classification. Moreover, it presents a proof of equivalence between the two conjectures when $p \equiv 3 \pmod{4}$, and it covers the recent work of Chen involving divisibility conditions for the size of orbits of $\M_k(\Fp)$. Lastly, we present a conjecture involving divisibility conditions that would imply strong approximation for generic levels.

\textit{Acknowledgments.} The author would like to thank his advisor Peter Sarnak for all his encouragement in the preparation of this thesis.

\section{{\it Foundations:} Markoff, Fricke \& Nielsen}\label{foundations}

\subsection{The Markoff Equation} \hfill

\par Consider the Markoff equation \[\mathbf{X}: x^2 + y^2 + z^2 = 3xyz.\]
Denote by $\mathbf{X}(\mathbb{Z})$ the set of integer solutions of $\mathbf{X}$. Observe that $(x, y, z)$ is in $\mathbf{X}(\Z)$ if and only if any of its coordinate permutations also is. Moreover, $\mathbf{X}$ can be viewed as a quadratic equation on, say, $z$, so that $(x, y, z) \in \mathbf{X}(\mathbb{Z})$ if and only if $(x, y, 3xy-z) \in \mathbf{X}(\mathbb{Z})$. These equivalences define orbits in the set of solutions $\mathbf{X}(\mathbb{Z})$. More precisely, let 
\[r(x, y, z) = (y, x, z), \text{ } s(x, y, z) = (x, z, y), \text{ } t(x, y, z) = (x, y, 3xy - z).\]
Then, the action of permuting coordinates, or replacing a coordinate by its second root in the quadratic equation $\mathbf{X}$, can be generated by a composition of $r, s,$ and $t$. Naturally, the orbit of $(x, y, z)$ is defined as the subset of solutions of $\mathbf{X}$ obtained by applying finite sequences of compositions of $r, s,$ and $t$ to $(x, y, z)$.

\par Observe that $r, s,$ and $t$ are all involutions, i.e. they are their own inverse. The operation $t$ is referred to as a Vieta involution, and although it seems to have a different nature from $r$ and $s$, it will become clear later that they can be viewed as part of the same phenomenon.

\par As a variety with integer coefficients, the Markoff equation can also be defined over $\mathbb{F}_p$, yielding a set of solutions $\mathbf{X}(\mathbb{F}_p)$. It is natural to ask whether or not the canonical map
\[\pi: \mathbf{X}(\mathbb{Z}) \to \mathbf{X}(\mathbb{F}_p)\]
given by reduction modulo $p$ is surjective. This is equivalent to the question of whether or not every solution of $\mathbf{X}$ in $\mathbb{F}_p$ descends from a solution of $\mathbf{X}$ in $\mathbb{Z}$, a phenomenon we refer to as strong approximation. In order to address this question, we notice that the functions $r, s,$ and $t$ can be defined in the same way over $\mathbb{F}_p$ and $\mathbb{Z}$. Therefore, this is equivalent to asking whether or not every orbit of $\mathbf{X}(\mathbb{F}_p)$ is in the image of an orbit of $\mathbf{X}(\mathbb{Z})$.

\par The orbit structure of $\mathbf{X}(\mathbb{Z})$ is well understood. In fact, a result of Markoff (c.f. \cite{markoff1880formes}) shows that $\mathbf{X}(\mathbb{Z})$ consists of two orbits: $\{(0, 0, 0)\}$ and $\mathbf{X}(\mathbb{Z}) \setminus \{(0, 0, 0)\}$. This amounts to saying that every element in $\mathbf{X}(\mathbb{Z}) \setminus \{(0, 0, 0)\}$ can be obtained by composing the functions $r, s,$ and $t$ to the initial solution $(1, 1, 1)$. Therefore, strong approximation is equivalent to showing that $\mathbf{X}(\mathbb{F}_p)$ consists of exactly two orbits, as $\{(0, 0, 0)\}$ is also an orbit in this case.

\par At this point, it is convenient to make a change of variables in $\mathbf{X}$. Consider the Markoff-type equation given by \[\mathbf{M}: x^2 + y^2 + z^2 = xyz.\]

\par There is a map $\mathbf{X}(\mathbb{Z}) \to \mathbf{M}(\mathbb{Z})$ given by $(x, y, z) \mapsto (3x, 3y, 3z)$. It is clear that this map is injective, and it is not hard to see it will also be surjective. Similarly, $\mathbf{X}(\mathbb{F}_p) \to \mathbf{M}(\mathbb{F}_p)$ also establishes a bijection for $p \ne 3$. In this context, Markoff's result classifies the orbits of $\mathbf{M}(\mathbb{Z})$, and strong approximation holds if and only if $\mathbf{M}(\mathbb{F}_p)$ consists of two orbits: $\{(0, 0, 0)\}$ and $\mathbf{M}(\mathbb{F}_p) \setminus \{(0, 0, 0)\}$. We now shift our focus to the orbits of $\M(\Fp)$, which carry a subtle algebraic structure.

\subsection{Fricke's Identity}\label{fricke} \hfill

\par The advantage of working with $\mathbf{M}$ over $\mathbf{X}$ is its moduli interpretation. This is a general phenomenon pertaining to the family of equations
\[\mathbf{M}_k : x^2 + y^2 + z^2 - xyz - 2 = k,\]
defined over an arbitrary field $\mathbb{F}$. Although our main focus will be to study these equations over $\Fp$, we present the ideas in this section in their original context, with no restrictions on the field $\mathbb{F}$.

\par Here, the Markoff-type equation $\M$ corresponds to $\M_{-2}$. For each $(x, y, z) \in \mathbb{F}^3$, there is a unique $k \in \mathbb{F}$ such that $(x, y, z) \in \M_k(\mathbb{F})$, which will be called the level of the triple $(x, y, z)$. Remarkably, the solutions for each level appear as traces of matrices in $\SL(\mathbb{F})$. More precisely, we have the following result due to Fricke (c.f. \cite{fricke1896theorie}):

{\lemma[Fricke's Identity]\label{fricke_id} Let $A, B$ be elements of $\SL(\mathbb{F})$. Then,
\[\tr(A)^2 + \tr(B)^2 + \tr(AB)^2 - \tr(A)\tr(B)\tr(AB) - 2 = \tr([A, B]),\]
where $[A, B]$ denotes the commutator of $A$ and $B$ in $\SL(\mathbb{F})$.}

\begin{proof}
We start with $A + A^{-1} = \tr(A)I$. By multiplying $AB$ on the left and $B^{-1}$ on the right, we obtain $ABAB^{-1} + [A, B] = \tr(A)A$. By taking traces, we get \[\tr(A)^2 - \tr(ABAB^{-1}) = \tr([A, B]).\]
\par Similarly $B + B^{-1} = \tr(B)I$, so that $ABAB^{-1} = \tr(B)ABA - (AB)^2$. By taking traces and using the fact that $\tr(X^2) = \tr(X)^2 - 2$ we get $\tr(ABAB^{-1}) = \tr(B)\tr(ABA) - \tr(AB)^2 + 2$. Therefore,
\[ \tr(A)^2 + \tr(AB)^2 - \tr(B)\tr(ABA) - 2 = \tr([A, B]).\]
\par Finally, notice that $ABA + ABA^{-1} = \tr(A)AB$. By taking traces and using that trace is invariant under conjugation, we obtain $\tr(ABA) = \tr(A)\tr(AB) - \tr(B)$. Hence
\[tr(A)^2 + \tr(B)^2 + \tr(AB)^2 - \tr(A)\tr(B)\tr(AB) - 2 = \tr([A, B]),\]
as desired. \end{proof}

\par This identity motivates the definition of the trace map $\Tr: \SL(\mathbb{F}) \times \SL(\mathbb{F}) \to \mathbb{F}^3$,
\[\Tr: (A, B) \mapsto (\tr(A), \tr(B), \tr(AB)).\]
With this notation, we have the following corollary:

\begin{cor}\label{fricke}
Let $(A, B)$ be a pair of elements in $\SL(\mathbb{F})$ such that $\tr([A, B]) = k$. Then, $\Tr(A, B)$ is at level $k$.
\end{cor}

\par The level of a pair $(A, B)$ is defined to be the level of the triple $\Tr(A, B)$, or equivalently, $\tr([A, B])$. Now, observe that $\Tr(A, B) = \Tr(A^X, B^X)$ for all $X$ in $\GL(\mathbb{F})$, as conjugation does not affect the trace. Moreover, the action generated by $r, s,$ and $t$ can be extended to $\mathbb{F}^3$, and the trace map adds a combinatorial meaning to it. This interpretation is derived from Nielsen moves, which we investigate next.

\subsection{Nielsen Moves}\label{Nielsen} \hfill

\par This section digresses from the context of $\SL(\mathbb{F})$ to introduce a general combinatorial notion that appears in group theory. Let $G$ be a group and $(A, B)$ a pair of elements in it. Denote by $H$ the subgroup generated by the pair $(A, B)$ in $G$. With no previous knowledge on the group $G$, we can still combinatorially derive other pairs of elements that generate the same subgroup $H$. For instance, the following involutions yield other generating pairs of $H$:
\[r(A, B) = (B, A), \text{ }s(A, B) = (A^{-1}, AB),\text{ } t(A, B) = (A^{-1}, B).\]

\par A Nielsen move is defined as any finite composition of the functions $r, s,$ and $t$. Although this is not the standard definition $-$ which only regards certain compositions of $r, s$, and $t$ as Nielsen moves $-$ this notion is more suited for our purposes.

\par We say that two pairs of elements $(A, B)$ and $(A', B')$ are Nielsen equivalent if there is a Nielsen move mapping one to the other. In fact, all possible changes in the generating set of a generic group $G$ are induced by Nielsen moves. This is described in the following result of Nielsen (c.f. \cite{nielsen1917isomorphismen}):

\begin{thm} Let $\Pi$ be the free group on two generators. For each Nielsen move $\mathcal{N}$, let $\Phi_{\mathcal{N}}$ be the automorphism of $\Pi$ defined by extending $\mathcal{N}$ multiplicatively from the pair of generators of $\Pi$. Then, for each automorphism $\Psi$ of $\Pi$, there is a Nielsen move $\mathcal{N}$ such that $\Psi = \Phi_{\mathcal{N}}$.
\end{thm}

\par From this, we derive an important corollary:

\begin{cor}\label{haha}
Let $G$ be a group and $A, B$ be in $G$. Denote by $H$ the subgroup of $G$ generated by the pair $(A, B)$ and let $X \in H$. Then, $(A, B)$ is Nielsen equivalent to $(A^X, B^X)$.
\end{cor}
\begin{proof}
Consider an expression for $X$ as a product of $A$, $B$, $A^{-1}$, and $B^{-1}$. For each term in this expression, replace $A$ by $a$ and $B$ by $b$, and let $x$ be the string obtained by this process. Let $\Pi$ be the free group on two generators, which we denote by $a$ and $b$, and let $x$ be the corresponding element of $\Pi$.

\par As conjugation by $x$ induces an automorphism $\Psi$ in $\Pi$, there is a Nielsen move $\mathcal{N}$ such that $\Psi = \Phi_{\mathcal{N}}$. Therefore $\mathcal{N}(a, b) = (a^x, b^x)$, and correspondingly we obtain $\mathcal{N}(A, B) = (A^X, B^X)$, so that the pairs $(A, B)$ and $(A^X, B^X)$ are Nielsen equivalent.
\end{proof}

\par Nielsen moves also satisfy an important invariant property. In order to see that, define an extended conjugacy class of a group $G$ to be any subset that is closed under conjugation and inverses. For a pair of elements $(A, B)$ in $G$, its extended conjugacy class is defined to be the extended conjugacy class of the commutator $[A,B]$. Now, looking at the commutators induced by the Nielsen moves $r, s,$ and $t$, we have
\[[B, A] = [A^{-1}, AB] = [A, B]^{-1}, \text{ } [A^{-1}, B] = ([A, B]^{-1})^{A^{-1}}.\]
As any Nielsen move is generated by $r, s,$ and $t$, it follows that two pairs $(A, B)$ and $(A', B')$ are Nielsen equivalent only if they have the same extended conjugacy class.

\par McCollough and Wanderley conjectured in \cite{mccullough2013nielsen} that for $\SL(\Fp)$, the extended conjugacy class of a generating pair is a complete Nielsen invariant. This conjecture is referred to as Q-classification and is closely related to the problem of strong approximation. This is studied in later sections of this paper, after developing a few technical tools.

\section{{\it Technical results:} Geometry \& Subgroups} \label{technical}

\subsection{Geometry of $\Fp$} \hfill

\par A useful geometric notion to have in mind when working with Markoff triples is that of hyperbolic, parabolic, and elliptic elements of $\Fp$. For $t \in \Fp$, we say that $t$ is hyperbolic when $t^2 - 4$ is a non-zero square. Similarly, we say that $t$ is parabolic when $t= \pm 2$ and elliptic when $t^2 - 4$ is not a square. In other words, $t$ is hyperbolic, parabolic, and elliptic when \[\left(\frac{t^2 - 4}{p}\right) = 1, 0,\text{ and }-1.\]

\par This classification allows us to write elements of $\Fp$ in a useful way:

\begin{lemma}\label{geometry}
Let $t \in \Fp$. Then, $t$ is hyperbolic if and only if $t = x + x^{-1}$ for some $x \in \Fp^{\times} \setminus \{\pm 1\}$. Moreover, $t$ is elliptic if and only if $t = y + y^{-1}$ for some $y \in \Fpp^{\times} \setminus \{\pm 1\}$ satisfying $y^{p+1} = 1$.
\end{lemma}

\begin{proof}
\par Assume first that $t$ is hyperbolic. Let $t^2 - 4 = a^2$, $a \in \Fp^{\times}$, and write $x = (t+a)/2$. Notice that $x \ne \pm 1$, otherwise $x = x^{-1}$ and $a = 0$. Then $x^{-1} = (t-a)/2$ and $t = x + x^{-1}$, $x \in \Fp^{\times} \setminus {\pm 1}$. Conversely, if $t = x + x^{-1}$, $x \ne \pm 1$, then $t^2 - 4 = (x - x^{-1})^2 \ne 0$ is in fact a non-zero square.

\par In the elliptic case, write $t^2 - 4 = a^2$, $a \in \Fpp^{\times} \setminus \{\pm 1\}$. As $a^2 \in \Fp$ but $a \not \in \Fp$, we have that $a^{2p} = a^2$ and $a^p = -a$. Write $y = (t + a)/2$, so that $y^{-1} = (t-a)/2$ and $t = y + y^{-1}$. Notice that $y \ne \pm 1$, otherwise $y = y^{-1}$ and $a = 0$. Moreover, $y^p = (t^p + a^p)/2^p = (t - a)/2 = y^{-1}$, hence $y^{p+1} = 1$. Conversely, if $t = y+ y^{-1}$, $y \ne \pm 1$, then $t^2 - 4 = (y - y^{-1})^2$. It is clear that $y - y^{-1}$ is not in $\Fp$, otherwise $y \in \Fp$ but $y^{p+1} = 1$, $y \ne \pm 1$. Therefore $t^2 - 4$ is not a square in $\Fp$.
\end{proof}

\subsection{Normal Pairs in $\SL(\Fp)\times \SL(\Fp)$} \hfill

\par These geometric notions naturally extend to matrices in $\SL(\Fp)$. We say that a matrix $X$ in $\SL(\Fp)$ is hyperbolic, parabolic, and elliptic depending on whether $\tr(X)$ is hyperbolic, parabolic, and elliptic. Notice that each conjugacy class of $\SL(\Fp)$ has constant trace, but also most conjugacy classes contain all elements of a given trace.

\par In order to see this phenomenon, it is useful to consider an embedding $\EL(\Fp)$ of $\SL(\Fp)$ into $\SL(\Fpp)$ that reveals some of the structure of elliptic elements. This embedding will be called the elliptic embedding of $\SL(\Fp)$ and is presented in the following lemma:

\begin{lemma}\label{emb}
Let $\EL(\Fp)$ be the subgroup of $\SL(\Fpp)$ containing all matrices of the form
\[\begin{pmatrix} x & y \\ y^p & x^p \end{pmatrix}, \text{ } x^{p+1} - y^{p+1} = 1.\]
Then, there is an isomorphism $\varphi: \SL(\Fp) \to \EL(\Fp), X \mapsto PXP^{-1}$, with $P \in \GL(\Fpp)$.

\end{lemma}

\begin{proof}
Let $g \in \Fpp \setminus \Fp$, so that $(g^p-g)^p = -(g^p - g) \ne 0$, and let
\[P = \begin{pmatrix} 1 & g \\ 1 & g^p \end{pmatrix} \in \GL(\Fpp).\]
Then, for $X \in \SL(\Fp)$, one has $PXP^{-1} \in \EL(\Fp)$. Moreover, \[|\SL(\Fp)| = |\EL (\Fp)| = p(p^2-1).\]
Therefore $X \mapsto PXP^{-1}$ induces an isomorphism $\SL(\Fp) \cong \EL(\Fp)$.
\end{proof}

\par Denote by $\phi = \varphi^{-1}: \EL(\Fp) \to \SL(\Fp)$ the inverse isomorphism. In order to see the relationship between trace and conjugacy classes in $\SL(\Fp)$, we refer to the following classical result (c.f. \cite{mccullough2013nielsen}):

\begin{prop}\label{conjug}
Let $A, B \in \SL(\Fp)$ be such that $\tr(A) = \tr(B) = t$. Then, if $t$ is non-parabolic, $A$ and $B$ are conjugate. Moreover, if $t$ is parabolic, then there are two non-trivial conjugacy classes of trace $t$. 
\end{prop}

\par This implies that conjugacy classes coincide with trace classes for every non-parabolic trace $t$. We can now choose representatives for these classes. For $t = u+u^{-1}$ hyperbolic, $u \in \Fp^{\times} \setminus \{\pm 1\}$, define \[D_u = \begin{pmatrix} u & 0 \\ 0 & u^{-1}\end{pmatrix} \in \SL(\Fp).\]

\par In the elliptic case, consider the elliptic embedding $\EL(\Fp)$ induced by the conjugation map $\varphi$. As trace is preserved under conjugation, it is clear that Proposition $\ref{conjug}$ also applies to $\EL(\Fp)$. For $t = v + v^{-1}$ elliptic, $v \in \Fpp^{\times} \setminus \{\pm 1\}$, $v^{p+1} = 1$, define \[D_v = \begin{pmatrix} v & 0 \\ 0 & v^{-1}\end{pmatrix} \in \EL(\Fp).\]

\par Hence the representative of a hyperbolic conjugacy class is taken to be $D_u$, an element of $\SL(\Fp)$. Similarly, the representative of an elliptic conjugacy class is taken to be $D_v$, an element of the elliptic embedding $\EL(\Fp)$. This is summarized as follows:

\begin{cor}\label{hyperelliptic} Let $X \in \SL(\Fp)$ and let $\tr(X) = t$. Then, when $t = u+ u^{-1}$ is hyperbolic, $X$ is conjugate to $D_u$ in $\SL(\Fp)$. Moreover, when $t = v + v^{-1}$ is elliptic, $X$ is conjugate to $D_v$ in $\EL(\Fp)$.
\end{cor}

\par The conjugation between the elements $X$ of $\SL(\Fp)$ and $D_v$ of $\EL(\Fp)$ is taken in the obvious sense, namely via the isomorphism $\SL(\Fp) \cong \EL(\Fp)$ induced by the conjugation map $\varphi$. We use this language to simplify notation whenever the context is clear. Corollary \ref{hyperelliptic} provides us with a convenient representation for pair of matrices $(A, B)$. In fact, we say that a pair $(A, B)$ is normal when
\[A = \begin{pmatrix} x & 0 \\ 0 & x^{-1}\end{pmatrix}, \text{ } B = \begin{pmatrix} a & b \\ c & d\end{pmatrix},\]
either as elements of $\SL(\Fp)$ or as elements of $\EL(\Fp)$. Corollary \ref{hyperelliptic} shows that any pair $(A, B)$ in $\SL(\Fp)$ with $A$ non-parabolic is conjugate to a normal pair $(D_x, C)$, which lies in $\SL(\Fp)$ when $A$ is hyperbolic and in $\EL(\Fp)$ when $A$ is elliptic.

\par Normal pairs provides us with a straightforward connection between the two matrices and the trace of its commutator. In particular, notice that this trace does not depend on the ambient space $\SL(\Fp)$ or $\EL(\Fp)$, as these are isomorphic via conjugation. Then, given
\[A = \begin{pmatrix} x & 0 \\ 0 & x^{-1}\end{pmatrix}, \text{ } B = \begin{pmatrix} a & b \\ c & d\end{pmatrix}, \text{ we have } [A, B] = \begin{pmatrix} ad-bcx^2 & ab(x^2-1) \\ cd(x^{-2}-1) & ad - bcx^{-2}\end{pmatrix},\]
hence $\tr([A, B]) = 2 - bc(\tr(A)^2 - 4)$. We have derived the following result:

\begin{lemma}\label{normal}
Let $(A, B)$ be a normal pair and let $k = \tr([A, B])$. Then, \[2-k = bc(\tr(A)^2 - 4).\]
\end{lemma}

\subsection{Subgroups of $\PSL(\Fp)$}\label{subgroup_class} \hfill

\par This subsection presents a classification of subgroups of $\PSL(\Fp)$ which will be used to study subgroups of $\SL(\Fp)$ through the canonical map $\rho: \SL(\Fp) \to \PSL(\Fp), X \mapsto \pm X$. We will often talk about the subgroup generated in $\PSL(\Fp)$ by a pair $(A, B)$ in $\SL(\Fp)$, which obviously means the subgroup generated by the images $(\rho(A), \rho(B))$. This abuse of language helps us simplify notation.

\par In this classification, each subgroup of $\PSL(\Fp)$ can fall into three different categories: affine, projective, and exceptional. These three categories are nearly disjoint, with a single intersection between affine and exceptional subgroups. We will return to this point later. For now, we introduce the classification: 

\begin{enumerate}
    \item[1.] {\it Affine Subgroups.}
 \par Let $H_1$ be the subgroup of $\SL(\Fp)$ consisting of lower triangular matrices: \[H_1 = \left \{ \begin{pmatrix} a & 0 \\ c & a^{-1} \end{pmatrix} : a \in \Fp^{\times}, \text{ }c \in \Fp \right\},\]
and let $A_1 = \rho(H_1)$ be the corresponding image in $\PSL(\Fp)$. Similarly, denote by $H_2$ the subgroup of $\EL(\Fp)$ consisting of diagonal matrices:
\[H_2 = \left \{ \begin{pmatrix} x & 0 \\ 0 & x^p \end{pmatrix} : x \in \Fp,\text{ } x^{p+1} = 1 \right\},\]
and let $A_2 = \rho \circ \phi(H_2)$. We will say that $A$ is an affine subgroup of $\PSL(\Fp)$ whenever $A$ is conjugate to a subgroup of either $A_1$ or $A_2$. \\

\item[2.] {\it Projective Subgroups.}
\par The only projective subgroup is $\PSL(\Fp)$ itself. It is worth remarking that this class becomes more complex when we study subgroups of $\PSL(\mathbb{F}_{p^n})$ for higher values of $n$. For this reason, we acknowledge it as a separate class. \\

\item[3.] {\it Exceptional Subgroups.}
\par This class is derived from triangle groups. The triangle group $(l, m, n)$ is defined as the abstract group with presentation
\[\langle a, b, c \text{ }|\text{ } a^l = b^m = c^n = abc = 1\rangle.\]
It is known that a triangle group $(l, m, n)$ is finite if and only if $\frac{1}{l} + \frac{1}{m} + \frac{1}{n} > 1$ (c.f. \cite{coxeter2013generators}). This gives us a small list of finite triangle groups which we show in Table \ref{exceptional_groups} below. A subgroup of $\PSL(\Fp)$ will be called exceptional if it is isomorphic to one of these finite triangle groups.
\begin{center}
\begin{table}[ht]
\caption{Exceptional Subgroups}
\begin{tabular}{ c  c }
 $(l, m, n)$ & Group  \\ [0.5ex]
 \hline
 $(2, 2, n)$ & Dihedral ($D_n$) \\
 $(2, 3, 3)$ & Tetahedral ($A_4$) \\
 $(2, 3, 4)$ & Octahedral ($S_4$) \\
 $(2, 3, 5)$ & Icosahedral ($A_5$)    
\end{tabular}
\label{exceptional_groups}
\end{table}
\end{center}
\end{enumerate}

\par For our purposes, it is important to understand when each of the exceptional subgroups appears as a subgroup of $\PSL(\Fp)$. When this happens, we say this exceptional subgroup is admissible in $\PSL(\Fp)$. Whether or not an exceptional group is admissible in $\PSL(\Fp)$ depend on local conditions on the prime $p$. This is given by the following important result (c.f. \cite{suzuki}):

\begin{thm}\label{admissibles}
Let $p$ be an odd prime and $\PSL(\Fp)$ the projective special linear group. Then,
\begin{itemize}
    \item[(i)] $D_n$ is admissible in $\PSL(\Fp)$ if and only if $n$ divides $p \pm 1$.
    \item[(ii)] $A_4$ is admissible in $\PSL(\Fp)$.
    \item[(iii)] $S_4$ is admissible in $\PSL(\Fp)$ if and only if $p \equiv \pm 1 \pmod{8}$.
    \item[(iv)] $A_5$ is admissible in $\PSL(\Fp)$ if and only if $p \equiv \pm 1 \pmod{5}$ or $p = 5$.
\end{itemize}
\end{thm}

\par In view of this classification, we say that a pair $(A, B)$ in $\SL(\Fp)$ has affine, projective, and exceptional group type according to the class of the subgroup generated in $\PSL(\Fp)$. As previously mentioned, these three categories are nearly disjoint. In fact, the only projective subgroup $\PSL(\Fp)$ is neither affine nor exceptional. Moreover, in Section \ref{main} we show that the only exceptional subgroup that is also affine is the trivial dihedral group $D_1$.

\par We remark that the local conditions on $p$ in Theorem \ref{admissibles} can be expressed in different ways. As our problem can be easily verified for small values of $p$, we assume throughout that $p > 5$. Then, the local conditions in $\text{(iii)}$ and $\text{(iv)}$ are equivalent to the Legendre identites $(\frac{2}{p}) = 1$ and $(\frac{5}{p}) = 1$, which we sometimes write simply as $\sqrt{2} \in \Fp$ and $\sqrt{5} \in \Fp$.

\section{{\it Macbeath:} The Structure of Towers} \label{macbeath}

\par This section closely follows the work of Macbeath, who studied generating pairs of the projective special liner group $\PSL(\Fp)$ in \cite{macbeath}. Our goal is to associate a type of subgroup of $\PSL(\Fp)$ to each triple in $\Fp^3$. Although we use a slightly different language, the essence of the proofs is still the same.

\subsection{The Trace Map} \hfill

\par Recall that over $\Fp$, the trace map $\Tr: \SL(\Fp) \times \SL(\Fp) \to \Fp^3$ is defined by
\[\Tr: (A, B) \mapsto (\tr(A), \tr(B), \tr(AB)).\]
As our notation suggests, the Nielsen moves $r, s,$ and $t$ previously defined correspond to the two coordinate permutations and the Vieta involution defined over $\Fp^3$. This correspondence is given by the trace map. In fact,
\[\Tr \circ r(A, B) = (\tr(B), \tr(A), \tr(BA)) = r \circ \Tr(A, B),\]
\[\Tr \circ s(A, B) = (\tr(A^{-1}), \tr(AB), \tr(B)) = s \circ \Tr (A, B),\]
\[\Tr \circ t(A, B) = (\tr(A^{-1}), \tr(B), \tr(A^{-1}B)) = t \circ \Tr(A, B).\]

\par In summary, upon identifying the maps $r, s,$ and $t$ as previously defined, we have
\[ \Tr \circ \mathcal{N} = \mathcal{N} \circ \Tr\]
for every Nielsen move $\mathcal{N}$. For this reason, we think of the action generated by $r, s,$ and $t$ taking place in both spaces $\Fp^3$ and $\SL(\Fp) \times \SL(\Fp)$. The analytic structure of the former allows a detailed study of its orbits, while the algebraic structure of the latter allows us to make precise conjectures in group theoretical terms. When there is no risk of confusion of the ambient space, we refer to this action as the Nielsen action and to its orbits as Nielsen orbits. In $\Fp^3$, these orbits correspond to the orbits previously defined. In $\SL(\Fp) \times \SL(\Fp)$, they correspond to the Nielsen equivalence classes.

\par In order to study the correspondence between the spaces $\SL(\Fp) \times \SL(\Fp)$ and $\Fp^3$, we consider the fibers \[\E(x, y, z) = \{(A, B) \in \SL(\Fp) \times \SL(\Fp) : \Tr(A, B) = (x, y, z)\}.\]
\par As previously remarked, $\E(x, y, z)$ is closed under conjugation by $\GL(\Fp)$. In fact, if $(A, B) \in \E(x, y, z)$, then $(A^X, B^X) \in \E(x, y, z)$ for every $X \in \GL(\Fp)$. We will refer to $\E(x, y, z)$ as the tower over $(x, y, z) \in \Fp^3$. In fact, towers can never be empty:

\begin{thm} For every $(x, y, z) \in \Fp^3$, the tower $\E(x, y, z)$ is non-empty.
\end{thm}

\begin{proof}
We prove that there is a pair of matrices $(A, B)$ of the form \[A = \begin{pmatrix} 0 & 1 \\ -1 & x \end{pmatrix}, \text{ } B = \begin{pmatrix} a & b \\ c & d \end{pmatrix},\] in the tower $\E(x, y, z)$. As $\tr(B) = y$, $\tr(AB) = z,$ and $\det(B) = 1$, we have \[a + d = y, \text{ } dx - b + c = z, \text{ } ad - bc = 1.\]
From the first two equations, we find $a$ and $b$ in terms of the other variables. By plugging this in the last equation, the system reduces to $d (y - d) - c(dx + c - z) = 1$, so that \[Q_{xyz}(-1, c, d) = 1 + c^2 + d^2 + xcd - yd - zc = 0, \text{ where }\]
\[Q_{xyz}(t, u, v) = t^2 + u^2 + v^2 + xuv + ytv + ztu.\]

\par As a ternary quadratic form, $Q_{xyz}$ defines a projective conic $\mathcal{C}$ which is non-empty in $\Fp^3$. If there is a point $(t_0, u_0, v_0)$ in $\mathcal{C}$ with $t_0 \ne 0$, we obtain a solution for the system by taking $c = -u_0/t_0, d = -v_0/t_0$. On the other hand, if $P = (0, u_0, v_0)$ is in $\mathcal{C}$ and $l$ is a line different from $t = 0$ through $P$, then $l$ intersects $\mathcal{C}$ in a second point $Q$. As $l$ only intersects $t = 0$ at $P$, either the first coordinate of $Q$ is non-zero or $P = Q$ is a singular point of $\mathcal{C}$. In the former case, we obtain a solution for the system. In the latter, we have $\partial_u Q_{xyz} (0, u_0, v_0) = \partial_v Q_{xyz}(0, u_0, v_0) = 0$, i.e.
\[ 2u_0 + xv_0 = 2v_0 + xu_0 = 0,\]
so that $x = \pm 2$. We have then showed that $\E(x, y, z)$ is non-empty unless possibly $x = \pm 2$. By applying the coordinate permutations $r$ and $s$ the same follows for all $\E(x, y, z)$, except possibly when $x, y, z = \pm 2$. In these cases, the system can be solved explicitly by $t = \pm u = -1$ and $v = 0$. The result follows.
\end{proof}

\par We have then showed that the trace map $\Tr$ is surjective. Our next goal is to investigate the Nielsen orbit structure of the tower $\E(x, y, z)$. From Corollary \ref{haha}, it follows that for each $X$ in the subgroup generated by $A$ and $B$, $(A^X, B^X)$ and $(A, B)$ are not only in the same tower but also in the same Nielsen orbit.

\par We now turn to the subgroup $H \le G$ generated by $A$ and $B$ in order to understand the Nielsen orbit structure of the tower $\E(x, y, z)$. Although a full classification of subgroups of $\SL(\Fp)$ is possible, it is sufficient and more convenient for us to work with the classification of subgroups of $\PSL(\Fp)$ introduced in Section \ref{subgroup_class}.

\subsection{Singular Triples} \hfill

\par Consider a triple $(x, y, z) \in \Fp^3$ and let $(A, B) \in \E(x, y, z)$. Then, by Corollary \ref{fricke}, \[x^2 + y^2 + z^2 - xyz - 2 = \tr([A, B]).\] \par As previously remarked, the orbit structure in $\E(x, y, z)$ is influenced by the subgroup of $\SL(\Fp)$ generated by $A$ and $B$. In particular, if $(A, B)$ is a generating pair of $\SL(\Fp)$, then the tower contains a large orbit. Remarkably, there is an entire level which contains no towers with generating pairs. This is described in the next result.

\begin{lemma}\label{trace_thm}
Let $A, B \in \SL(\Fp)$ be such that $\tr([A, B]) = 2$. Then, $(A, B)$ does not generate $\SL(\Fp)$.
\end{lemma}

\begin{proof}
As $\tr([A, B]) = 2$, we know that $[A, B]$ is conjugate to \[t_{\mu} = \begin{pmatrix} 1 & 0 \\ \mu & 1 \end{pmatrix}\]
for some $\mu \in \Fp$. As conjugation does not affect whether or not $(A, B)$ generates $\SL(\Fp)$, we can assume that $[A, B] = t_{\mu}$. If $\mu = 0$, then $A$ and $B$ commute and therefore cannot generate $\SL(\Fp)$. Otherwise, write \[B = \begin{pmatrix} a & b \\ c & d \end{pmatrix}, \text{ so that } ABA^{-1} = t_{\mu}B = \begin{pmatrix} a & b \\ c + \mu a & d + \mu b \end{pmatrix}.\]
\par By taking traces on both sides we obtain $b = 0$, so that $B$ is lower triangular. However, as $[A, B] = t_{\mu}$, it follows that $A$ is also lower triangular, so that $A$ and $B$ cannot generate $\SL(\Fp)$. 
\end{proof}

\par Motivated by Lemma \ref{trace_thm}, we refer to triples at level $k = 2$ as singular triples. This notation extends to pairs $(A, B)$ of $\SL(\Fp)$, which we call singular when $\Tr(A, B)$ is singular, or equivalently, when $\tr([A, B]) = 2$. Our first goal is to classify singular triples according to the geometry of its elements. By considering the Cayley cubic of $\M_{2}(\Fp)$, \[x^2 + y^2 + z^2 - xyz = 4,\] as a quadratic equation in $x$, we obtain the discriminant $\Delta_x = (y^2-4)(z^2-4)$, which must be a square in $\Fp$. Similarly, $\Delta_y = (x^2-4)(z^2-4)$ and $\Delta_z = (x^2-4)(y^2-4)$ must also be squares. Hence a singular triple cannot have coordinates, say $x$ and $y$, which are hyperbolic and elliptic, for otherwise
\[\left(\frac{\Delta_z}{p}\right) = \left(\frac{x^2-4}{p}\right)\left(\frac{y^2-4}{p}\right)
= - 1,\] and $\Delta_z$ is not a square. Moreover, it is clear that if two coordinates of a singular triple are parabolic, then so is the third one. Therefore, we arrive at the following result:

\begin{lemma}\label{hyperb_triples}
A singular triple $(x, y, z)$ cannot have both a hyperbolic and an elliptic coordinate. Moreover, exactly one of the following conditions is true:\begin{enumerate}
    \item[(i)] $(x, y, z)$ has at least two hyperbolic coordinates.
    \item[(ii)] $(x, y, z)$ has at least two elliptic coordinates.
    \item[(iii)] $(x, y, z)$ has three parabolic coordinates.
\end{enumerate}
\end{lemma}

This allows us to classify singular triples $(x, y, z)$ into hyperbolic, elliptic, and parabolic, according to the most abundant geometric type of its coordinates. We now turn to the main goal of this section, which is to classify the group type of singular pairs $(A, B)$.

\begin{thm}\label{singular}
Let $(A, B)$ be a pair of elements of $\SL(\Fp)$. Then, $(A, B)$ is singular if and only $A$ and $B$ generate an affine subgroup of $\PSL(\Fp)$.
\end{thm}

\begin{proof}
First, assume that $(A, B)$ is singular, so that the triple $\Tr(A, B) = (x, y, z)$ is also singular. By Lemma \ref{hyperb_triples}, $(x, y, z)$ can either be hyperbolic, elliptic, or parabolic. We consider these three cases:

\begin{enumerate}
    \item[(i)] {\it Hyperbolic.} Notice that coordinate permutations only induce a Nielsen move on the pair $(A, B)$, therefore it does not change the subgroup generated by the pair. Hence, we can assume that $x$ and $y$ are hyperbolic. Since $A$ is hyperbolic and conjugation does not affect the type of group generated by $A$ and $B$, by Corollary \ref{hyperelliptic} we assume \[A = \begin{pmatrix} u & 0 \\ 0 & u^{-1} \end{pmatrix},\text{ } B = \begin{pmatrix} a & b \\ c & d \end{pmatrix} \in \SL(\Fp).\]
    \par This is a normal pair in $\SL(\Fp)$ lying at level $k = 2$. By Lemma \ref{normal}, it follows that $bc(x^2 - 4) = 0$. As $x$ is hyperbolic, either $b = 0$ or $c = 0$, so that $(A, B)$ generates a subgroup of either the lower triangular group $H_1 \le \SL(\Fp)$, or of its conjugate, the upper triangular group 
    \[H_1^Q \le \SL(\Fp),\text{ where }Q = \begin{pmatrix} 0 & 1 \\ -1 & 0 \end{pmatrix}.\]
    In either case, $(A, B)$ generates an affine subgroup of $\PSL(\Fp)$.
    \vspace{3pt}
    
    \item[(ii)] {\it Elliptic.} As before, we can assume that the coordinates $x$ and $y$ are elliptic, so that $A$ and $B$ are both elliptic. The isomorphism $\SL(\Fp) \cong \EL(\Fp)$ induced by the conjugation map $\varphi$ allows us to assume that $A$ is a diagonal matrix, so that
    \[ A = \begin{pmatrix} v & 0 \\ 0 & v^{-1} \end{pmatrix},\text{ } B = \begin{pmatrix} a & b \\ b^p & a^p \end{pmatrix} \in \EL(\Fp).\]
    \par This is a normal pair in $\EL(\Fp)$ lying at level $k = 2$, hence by Lemma \ref{normal} it follows that $b^{p+1}(x^2-4) = 0$. As $x$ is elliptic, $b = 0$, so that $B$ is a diagonal matrix in $\EL(\Fp)$. Hence $A, B \in H_2$ and $(A, B)$ generates an affine subgroup of $\PSL(\Fp)$.
    \vspace{3pt}
    
    \item[(iii)] {\it Parabolic.} 
    Notice that all singular parabolic triples are in the same orbit as $(2, 2, 2)$. As Nielsen moves do not change the subgroup generated by the pair $(A, B)$, we can assume that $\Tr(A, B) = (2, 2, 2)$. If $A, B$ are both the identity, then the result is immediate. Otherwise, upon possibly exchanging $A$ and $B$, we can assume $A \ne I$. As $A$ is parabolic, conjugation allows us to assume that
    \[A = \begin{pmatrix} 1 & 0 \\ t & 1 \end{pmatrix},\text{ } B = \begin{pmatrix} a & b \\ c & d \end{pmatrix} \in \SL(\Fp).\]
    \par Using $\tr(B) = \tr(AB) = 2$, we obtain $a + d = a+d + bt$. Since $A \ne I$, we must have $b = 0$. Therefore $A, B$ are lower triangular matrices in $H_1$, and therefore $(A, B)$ generates an  affine subgroup of $\PSL(\Fp)$.

    \end{enumerate}

\par Conversely, assume that $(A, B)$ generates an affine subgroup of $\PSL(\Fp)$. Then, as $\rho^{-1}(A_1) = H_1 \le \SL(\Fp)$ and $(\rho \circ \phi)^{-1}(A_2) = H_2 \le \EL(\Fp)$, we can assume that $A$ and $B$ have a lower diagonal representation as matrices, either in $\SL(\Fp)$ or in $\EL(\Fp)$. Let
\[A = \begin{pmatrix} u & 0 \\ * & u^{-1} \end{pmatrix},\text{ } B = \begin{pmatrix} v & 0 \\ * & v^{-1} \end{pmatrix}.\]
Therefore $x = u + u^{-1}, y = v + v^{-1},$ and $z = uv + u^{-1}v^{-1}$, so that $x^2 + y^2 + z^2 - xyz - 2 = 2$. It follows that $(A, B)$ is a singular pair.
\end{proof}

\par We have showed that singular pairs $(A, B)$ are somewhat special. In fact, any singular pair generates an affine subgroup of $\PSL(\Fp)$, and any pair that generates an affine subgroup of $\PSL(\Fp)$ is singular. This understanding suffices for the singular case and we will not pursue Nielsen orbit structures when $k = 2$. In fact, the Cayley cubic $\M_{2}(\Fp)$ exhibits a linear behavior different from all other levels, making its treatment unique.

\par We are then left with pairs $(A, B)$ which generate projective or exceptional subgroups of $\PSL(\Fp)$ that are not affine, and therefore cannot lie at level $k = 2$. For a non-singular triple $(x, y, z) \in \M_k(\Fp)$, $k\ne 2$, the tower $\E(x, y, z)$ only contains pairs that generate projective or exceptional subgroups of $\PSL(\Fp)$. Our next goal is to investigate the Nielsen orbit structure of these towers.

\subsection{Conjugation of Towers} \hfill

\par Recall that towers containing generating pairs will likely have large Nielsen orbits as conjugation by every element $X$ of $\SL(\Fp)$ is induced by a Nielsen move, and moreover this conjugation could only stabilize a generating pair $(A, B)$ if $X = \pm I$, since $\SL(\Fp)$ has center $\{ \pm I\}$.

\par Interestingly, this accounts for most of the Nielsen orbit structure of a tower containing a generating pair. The remarkable feature here is that given a non-singular triple $(x, y, z)$, the subgroup generated by any pair $(A, B) \in \E(x, y, z)$ is, upon conjugation, fully determined by $(x, y, z)$. In fact, we have the following result:

\begin{thm}\label{towers}
Let $(x, y, z)$ be a non-singular triple. Then, the tower $\E(x, y, z)$ contains one conjugacy class in $\GL(\Fp)$ and two conjugacy classes in $\SL(\Fp)$. 
\end{thm}

\begin{proof}
Notice that the desired property is invariant under Nielsen moves. In other words, given $(x, y, z)$ a non-singular triple for which $\E(x, y, z)$ contains one conjugacy class in $\GL(\Fp)$ and two conjugacy classes in $\SL(\Fp)$ and $\mathcal{N}$ an arbitrary Nielsen move, the tower $\E(\mathcal{N}(x, y, z))$ also contains one $\GL(\Fp)$ conjugacy class and two $\SL(\Fp)$ conjugacy classes. In fact, $\mathcal{N}$ creates a natural bijection between $\E(x, y, z)$ and $\E(\mathcal{N}(x, y, z))$ in which $(A, B)$ corresponds to $\mathcal{N}(A, B)$. This bijection commutes with conjugation so that $\mathcal{N}((A, B)^X) = (\mathcal{N}(A, B))^X$ for all $X$ in $\GL(\Fp)$. Hence the two towers share the same conjugacy structure in $\GL(\Fp)$ and $\SL(\Fp)$.

\par We now prove the statement for at least one triple $(x, y, z)$ in each Nielsen orbit. First, assume that at least one of the coordinates is hyperbolic. Via Nielsen moves, we can assume that this coordinate is $x$, with $x = u + u^{-1}$. Then, for any pair $(A, B) \in \E(x, y, z)$, we can find an element $C \in \SL(\Fp)$ such that $(A, B)$ is conjugate to $(D_u, C)$, where
\[D_u = \begin{pmatrix} u & 0 \\ 0 & u^{-1} \end{pmatrix},\text{ } C = \begin{pmatrix} a & b \\ c & d \end{pmatrix} \in \SL(\Fp).\]
\par The idea now is to look at different elements $C$ for which $(D_u, C) \in \E(x, y, z)$ and check the conjugation of $(D_u, C)$ for a given $C$. For fixed $\tr(C) = y$ and $\tr(D_uC) = z$, the value of $(a, d)$ is determined by $(x, y, z)$, and conversely any $(a, d)$ satisfying such equations yield a matrix $C$ such that $(D_u, C) \in \E(x, y, z)$. Moreover, $C$ always satisfies $bc \ne 0$, otherwise the pair $(D_u, C)$ would be singular as its elements would be contained either in the upper or in the lower triangular group, and hence generate an affine subgroup in $\PSL(\Fp)$.

\par This shows that there are exactly $p-1$ pairs $(D_u, C)$ in $\E(x, y, z)$, one for each pair $(b, c)$ with a fixed non-zero product. Now, observe that the centralizer of $D_u$ in $\GL(\Fp)$ is given by all diagonal matrices. Hence, notice that for
\[X = \begin{pmatrix} t & 0 \\ 0 & 1 \end{pmatrix} \text{ and } C = \begin{pmatrix} a & b \\ c & d \end{pmatrix}, \text{ we have } C^X = \begin{pmatrix} a & bt \\ ct^{-1} & d \end{pmatrix} \in \SL(\Fp),\] therefore all pairs $(D_u, C)$ in $\E(x, y, z)$ are conjugate in $\GL(\Fp)$. In $\SL(\Fp)$, given
\[X = \begin{pmatrix} t & 0 \\ 0 & t^{-1} \end{pmatrix} \text{ and } C = \begin{pmatrix} a & b \\ c & d \end{pmatrix}, \text{ we have } C^X = \begin{pmatrix} a & bt^2 \\ ct^{-2} & d \end{pmatrix} \in \SL(\Fp),\]
so that two pairs $(D_u, C)$ in $\E(x, y, z)$ are conjugate in $\SL(\Fp)$ if and only if their corresponding values of $b$ are both squares or non-squares in $\Fp$. As every pair $(A, B)$ in $\E(x, y, z)$ is conjugate to a pair $(D_u, C)$, it follows that $\E(x, y, z)$ contains one $\GL(\Fp)$ conjugacy class and two $\SL(\Fp)$ conjugacy classes.

\par The same proof still works under the assumption that at least one of the coordinates in $(x, y, z)$ is elliptic, the only essential difference being that instead of working with normal pairs in $\SL(\Fp)$, we consider normal pairs in the elliptic embedding $\EL(\Fp)$.

\par We are left with the case when all coordinates in $(x, y, z)$ are parabolic and $(x, y, z)$ is non-singular. There are only four such triples, and it is easy to check that a non-parabolic coordinate $t = 6$ appears in the orbit of $(x, y, z)$. The result follows. \end{proof}

\par In the proof above, we omitted the details of the elliptic case by taking the hyperbolic case as our model. We remark that the inverse strategy is used by Macbeath (c.f. \cite{macbeath}), thus covering both cases in detail. With this result, we can deduce that all towers in fact have the same height:

\begin{cor}
Let $(x, y, z)$ be a non-singular triple. Then $|\E(x, y, z)| = p(p^2-1)$.
\end{cor}

\par This follows from the previous proof, noticing that the only elements of $\SL(\Fp)$ that commute with a normal pair $(D_u, C)$ are $\pm I$. Now, let $(x, y, z)$ be a non-singular triple and $(A, B)$ be in the tower $\E(x, y, z)$. Then, Theorem \ref{towers} shows that
\[\E(x,y,z) = \{ (A, B)^X : X \in \GL(\Fp) \}.\]
In particular, the conjugacy class of a subgroup generated by a pair in $\E(x, y, z)$ is determined by $(x, y, z)$. Moreover, as Nielsen moves do not change the subgroup generated by a pair $(A, B)$, it follows that every orbit of triples in $\Fp^3$ has associated to it a single conjugacy class in $\GL(\Fp)$ of subgroups of $\SL(\Fp)$.

\par Let $\Orb$ be an orbit of non-singular triples at level $k$ and let $H$ be one of its corresponding subgroups in $\SL(\Fp)$. Then, for $(x, y, z)$ in $\mathcal{O}$ and $(A, B)$ in $\E(x, y, z)$, we have that $(A, B)$ generates $H^X \le \SL(\Fp)$ for some $X$ in $\GL(\Fp)$. As $-I \in H$ if and only if $-I \in H^X$, the isomorphism class of the image of $H$ under the canonical map $\rho: \SL(\Fp) \to \PSL(\Fp)$ is well-defined as a type of subgroup of $\PSL(\Fp)$.

\par Therefore, we have associated an isomorphism class of subgroups of $\PSL(\Fp)$ to each non-singular orbit $\Orb$ at level $k$. By Theorem \ref{singular}, this is either a projective or an exceptional subgroup of $\PSL(\Fp)$. When the group is projective, the pair $(A, B)$ generates $\PSL(\Fp)$. Moreover, notice that $\PSL(\Fp)$ cannot be embedded in $\SL(\Fp)$ as the latter has only one element of order two (namely, $-I$), and $\PSL(\Fp)$ has several (in fact, the image under $\rho$ of any element of trace zero). Therefore, $(A, B)$ also generates $\SL(\Fp)$.

\par We conclude that a non-singular orbit $\Orb$ is either associated to an exceptional subgroup of $\PSL(\Fp)$, or any pair $(A, B)$ in one of its towers is in fact a generating pair of $\SL(\Fp)$. When $\Orb$ is of the latter type, we refer to it as a generating orbit, and otherwise we refer to it as a exceptional orbit.

\par We know that all non-singular levels $k$ are a disjoint union of generating and exceptional orbits. The problem of strong approximation aims to find exactly this orbit decomposition and in particular prove that there is a unique generating orbit at each \mbox{level $k$}. The classification of all exceptional orbits will gives us a precise analytic formulation of this statement. We investigate this next.

\section{{\it Classification:} Exceptional Orbits}\label{main}

\par In his work, Macbeath gives a brief account on conditions that a pair $(A, B)$ generating an exceptional subgroup of $\PSL(\Fp)$ must satisfy. These conditions are not directly related to the orbit of $\Tr(A, B)$, but instead with the orders of $A$ and $B$ as elements of $\PSL(\Fp)$. We now expand on those ideas to identify all the exceptional orbits.

\par Let $(A, B)$ be a pair in a tower over an exceptional orbit and let $H \le \PSL(\Fp)$ be the subgroup generated by $(A, B)$. Hence, $H$ is isomorphic to either $D_n$, $A_4$, $S_4$, or $A_5$. We now attempt to describe this exceptional orbit. Roughly speaking, the idea is to look at order of elements in generating pairs of exceptional groups, which provides us with the order of $A$ and $B$ in $\PSL(\Fp)$, and relate this with the traces of $A$ and $B$ in $\Fp$. This argument will be carried out separately for each exceptional subgroup of $\PSL(\Fp)$.

\subsection{Order \& Trace}\hfill

\par In order to proceed, it is important to have a way of relating the order of an element $X$ in $\PSL(\Fp)$ with its corresponding trace in $\SL(\Fp)$. This is given by the following lemma:
 
\begin{lemma}\label{order-trace}
Let $X$ be in $\SL(\Fp)$ and let $h$ be its order in $\PSL(\Fp)$.
\begin{itemize}
    \item[(i)] If $h = 2$, then $\tr(X) = 0$.
    \item[(ii)] If $h = 3$, then $\tr(X) = \pm 1$.
    \item[(iii)] If $h = 4$, then $\tr(X) = \pm \sqrt{2}$. 
    \item[(iv)] If $h = 5$, then $\tr(X) = \frac{\pm 1 \pm \sqrt{5}}{2}$.
\end{itemize}
\end{lemma}

\begin{proof}
We can assume that $h < p$ in the cases above. As $h \ne 1$ and $p$ does not divide the order of $X$ in $\SL(\Fp)$, we know that $X$ is non-parabolic. In particular, $X$ is conjugate to a diagonal matrix $D_{u}$ for some $u \in \Fpp$. Observe that as trace is invariant under conjugation, the matrices $X$ and $D_u$ have the same trace in $\Fp$. Notice also that $X^h = \pm I$ in $\SL(\Fp)$, so that $D_u^h = \pm I$ as well. Moreover, we remark that the only element of order two in $\SL(\Fp)$ is $-I$. We now consider each case separately.
\begin{itemize}
    \item[(i)] If $h = 2$, either $D_u^2 = I$ or $D_u^2 = -I$. In the first case, $D_u = \pm I$, but then $h$ is actually one, as $X$ maps to $\pm I$ in $\PSL(\Fp)$. In the second case, $D_u^2 = -I$, so that $u^2 = -1$ and $\tr(X) = \tr(D_u) = u + u^{-1} = 0$.
    \item[(ii)] If $h = 3$, we have $D_u^3 = \pm I$, so that $u^3 = \pm 1$. As $D_u \ne \pm I$, we must have $u \ne \pm 1$, and so $u^2 \pm u + 1 = 0$. Therefore $\tr(X) = \tr(D_u) = u + u^{-1} = \pm 1$.
    \item[(iii)] If $h = 4$, either $D_u^4 = I$ or $D_u^4 = -I$. In the first case, we must have $D_u^2 = \pm I$, but then $h$ equals at most two. Hence $D_u^4 = -1$ and $u^4 = -1$, so that $(u^2 -\sqrt{2} u  + 1)$ $(u^2 + \sqrt{2} u + 1) = 0$ and $u^2 + 1 = \pm \sqrt{2} u$. Therefore $\tr(X) = \tr(D_u) = u + u^{-1} = \pm \sqrt{2}$.
    \item[(iv)] If $h = 5$, we have $D_u^5 = \pm I$, so that $u^5 = \pm 1$. As $u \ne \pm 1$, we have $u^4 \pm u^3 + u^2 \pm u + 1 = 0$. Let $\phi_1, \phi_2$ denote the roots of $x^2 \pm x - 1$. Then, $(u^2 + \phi_1 u + 1)(u^2 + \phi_2 u + 1) = 0$, so that $\tr(X) = \tr(D_u) = u + u ^{-1} = \frac{\pm 1 \pm \sqrt{5}}{2}$.
    \end{itemize}
This finishes the proof. \end{proof}

We remark that the existence of elements of certain orders in $\PSL(\Fp)$ is dependent on local conditions on $p$. In fact, Lemma \ref{order-trace} shows that $\PSL(\Fp)$ can only have an element of order four when $\sqrt{2} \in \Fp$, that is, $(\frac{2}{p}) = 1$. In this case, Theorem \ref{admissibles} shows that $\PSL(\Fp)$ contains a subgroup isomorphic to $S_4$, which in fact has an element of order four. Moreover, Lemma \ref{order-trace} shows that $\PSL(\Fp)$ contains an element of order five only if $\sqrt{5} \in \Fp$, that is, $(\frac{5}{p}) = 1$. Theorem \ref{admissibles} shows that in that case, $\PSL(\Fp)$ contains a subgroup isomorphic to $A_5$, thus in fact containing an element of order five.

\par Notice that elements of order two and three always exist in $\PSL(\Fp)$ as six divides the order of $\PSL(\Fp)$. Now, with Lemma \ref{order-trace} in hands, we are able to travel through pairs of elements in exceptional subgroups via Nielsen moves and relate them with a corresponding exceptional orbit. The key feature here is that the orders of elements in generating pairs of exceptional subgroups lie in a very small set. We now investigate each case separately.

\newpage

\subsection{Dihedral ($D_n$)} \hfill

\par Assume that $D_n \le \PSL(\Fp)$ for some value of $n$. By Theorem \ref{admissibles}, we know that $n$ divides $p \pm 1$. Consider the standard presentation of $D_n$ given by
\[D_n = \langle x, y \text{ } | \text{ } x^n = y^2 = 1, yxy = x^{-1} \rangle.\]
Every element of $D_n$ can be uniquely written as $x^{r}y^{s}$, with $0 \le r \le n - 1$ and $0 \le s \le 1$, and moreover $x^ry$ has order two for every value of $r$. Notice that if $g_1, g_2$ generate $D_n$ and $g_3 = g_1g_2$, then at most one element among $g_1, g_2, g_3$ is a power of $x$. Therefore at least two elements in the triple $(g_1, g_2, g_3)$ have order two, and by Lemma \ref{order-trace} must have trace zero as elements of $\SL(\Fp)$.

\par Therefore, given a pair $(A, B)$ in $\SL(\Fp)$ generating a dihedral group $D_n$ in $\PSL(\Fp)$, it follows that at least two elements in the tripe $\Tr(A, B)$ are zero. This triple lies in a very small orbit of size at most six, in which two of the coordinates are always zero. Moreover, after possibly applying a Nielsen move, we assume that $\Tr(A, B) = (t, 0, 0)$ for some $t \in \Fp$, which will be taken as a representative of the orbit. Notice that this triple can only appear at level $k$ when $t^2 - 2 = k$, that is,
\[\left(\frac{k+2}{p}\right) = 0, 1.\]

\par In the singular case $k = 2$, we have $\Tr(A, B) = (\pm 2, 0, 0)$ which implies that either $p$ divides the order of $A$ in $\SL(\Fp)$, or $A = \pm I$. By Theorem \ref{admissibles}, the former cannot happen otherwise $D_n$ would have an element of order $p$ but $p$ does not divide $2(p \pm 1)$. Hence $A = \pm I$, and $B^2 = \pm I$. Hence, the subgroup generated by $(A, B)$ in $\PSL(\Fp)$ is isomorphic to $D_1$, which is in fact an affine subgroup of $\PSL(\Fp)$.

\par It remains to prove that for every level $k \ne 2$ for which $k+2 = t^2$ is a square, there is a pair $(A, B)$ which generates a dihedral group $D_n$ in $\PSL(\Fp)$ and satisfies $\Tr(A, B) = (t, 0, 0)$. By Theorem $\ref{admissibles}$, we know that $D_{p\pm 1} \le \PSL(\Fp)$. Consider the presentation \mbox{of $D_{p\pm 1}$} with two generators $x, y$, and let $A, B$ be elements of $\SL(\Fp)$ which map to $x, y$ in $\PSL(\Fp)$. Notice that $A$ cannot be a parabolic element, as that would imply that its order in $\SL(\Fp)$ is either $1$, $2$, $p$, or $2p$, contradicting the fact that $x$ has order $p\pm 1$ in $\PSL(\Fp)$. Therefore, there is $u \in \Fpp$ such that $A$ is conjugate to the diagonal matrix $D_u$, either in $\SL(\Fp)$ or in the elliptic embedding $\EL(\Fp)$.

\par For $D_{p-1}$, $A$ will have order divisible by $p-1$, which implies $u \in \Fp$ is a generator of $\Fp^{\times}$. In particular, $D_u$ lies in $\SL(\Fp)$. For $D_{p+1}$, $A$ will have order divisible by $p+1$, hence $u \not \in \Fp$ is a generator of the set of solutions for $z^{p+1} = 1$ in $\Fpp$. In this case, $D_u$ lies in the elliptic embedding $\EL(\Fp)$. Therefore, for every $X \in \Fp$ non-parabolic, $X$ is conjugate to a power of $D_u$. Moreover, notice that for each $1 \le d < p \pm 1$, the subgroup $\langle x^d, y \rangle \le D_{p\pm 1}$ is again dihedral.

\par Therefore, $(A^d, B)$ generates a dihedral subgroup in $\PSL(\Fp)$, and as the image $x^dy$ of $A^dB$ has order two, we obtain $\Tr(A^d, B) = (t, 0, 0)$ for some $t \in \Fp$. Moreover, as \mbox{$d$ varies} for both groups $D_{p-1}, D_{p+1}$, the element $A^d$ ranges over all non-parabolic conjugacy classes. In particular, $A^d$ reaches all non-parabolic trace values. Thus, for each $t \ne 2$, there exists a pair matrices $(A_0, B)$ in $\SL(\Fp)$ which generates a dihedral group in $\PSL(\Fp)$. This finishes the classification of exceptional orbits associated with dihedral subgroups $D_n$.

\subsection{Tetahedral ($A_4$)} \hfill

\par Using at the standard way of writing permutation groups, notice that $A_4$ only contains elements of order one, two, and three. Let $(A, B)$ be a pair in $\SL(\Fp)$ which generates $A_4$ in $\PSL(\Fp)$. Then, the images of $A$, $B$, and $AB$ have orders two and three in $\PSL(\Fp)$. Hence by Lemma \ref{order-trace}, $\Tr(A, B)$ has coordinates in $\{0, \pm 1\}$.

\par Notice that $A_4$ is not a dihedral group, and by our previous argument at least two of the coordinates in $\Tr(A, B)$ are non-zero. Then, $\Tr(A, B)$ is of the form
\[(\pm 1, \pm 1, \pm 1),\text{ or }(\pm 1, \pm 1, 0).\]
\par Moreover, notice that $A_4$ is not an affine subgroup. In fact, for every generating pair $(A, B)$ of $A_4$ in $\PSL(\Fp)$, we have that $[A, B]$ has order two or three in $\PSL(\Fp)$, and so $\tr([A, B]) = 0, \pm 1$. In particular, $\tr([A, B]) \ne 2$, therefore $(A, B)$ is not an affine pair and cannot generate an affine subgroup.

\par We are left with a single orbit at level $k = 0$ containing $(1, 1, 0)$ and fifteen other triples. As $A_4$ is always admissible, this is the only exceptional orbit associated with $A_4$.

\subsection{Octahedral ($S_4$)} \hfill

\par Assume that $S_4$ is admissible in $\PSL(\Fp)$, or equivalently, \mbox{$\sqrt{2} \in \Fp$}. Notice that $S_4$ only contains elements of order one, two, three, and four. Let $(A, B)$ be a pair in $\SL(\Fp)$ which generates $S_4$ in $\PSL(\Fp)$. Then, $[A, B]$ has order two, three, or four in $\PSL(\Fp)$, and by Lemma \ref{order-trace}, $\tr([A, B]) \in \{0, \pm 1, \pm \sqrt{2}\}$. In particular, $\tr([A, B]) \ne 2$, hence $S_4$ is not an affine subgroup of $\PSL(\Fp)$.

\par Notice that the images of $A$, $B$, and $AB$ cannot have order one in $\PSL(\Fp)$, therefore the coordinates of $\Tr(A, B)$ belong to $\{0, \pm 1, \pm \sqrt{2}\}$. Moreover, by our previous argument combined with the fact that $S_4$ is not a dihedral group, at least two coordinates in $\Tr(A, B)$ are non-zero. By Theorem \ref{towers}, we also know that the exceptional orbits associated with $A_4$ and $S_4$ must be disjoint. Therefore, $\Tr(A, B)$ is of the form
\[(\pm \sqrt{2}, \pm \sqrt{2}, \pm \sqrt{2}),\text{ } (\pm \sqrt{2}, \pm 1, \pm 1), \text{ } (\pm \sqrt{2}, \pm \sqrt{2}, \pm 1), \text{ or } (\pm \sqrt{2}, \pm 1, 0).\]

\par The triples $(\pm \sqrt{2}, \pm \sqrt{2}, 0)$ were omitted because they lie at the affine level $k = 2$. Now, using a Vieta move, we obtain coordinates $\pm 2 \pm \sqrt{2}$ and $\pm 1 \pm \sqrt{2}$ from the first two types of triples. These values, however, do not lie in the set $\{0, \pm 1, \pm \sqrt{2}\}$, and therefore cannot be obtained from a generating pair of $S_4$ in $\PSL(\Fp)$. Similarly, a few triples of the third type lie in the orbit of a triple with a coordinate equal to $\pm 3$, and hence cannot be part of an exceptional orbit associated with $S_4$.

\par We are left with an orbit at level $k = 1$ containing $(\sqrt{2}, 1, 0)$ and thirty-five \mbox{other triples}. When $S_4$ is admissible, that is, when $\sqrt{2} \in \Fp$, this is the only exceptional orbit associated with $S_4$.

\subsection{Icosahedral ($A_5$)} \hfill

\par Assume that $A_5$ is admissible in $\PSL(\Fp)$, or equivalently, \mbox{$\sqrt{5} \in \Fp$}. Notice that $A_5$ contains elements of order one, two, three, and five, therefore by Lemma \ref{order-trace} it follows that $A_5$ is not an affine subgroup of $\PSL(\Fp)$. Let $(A, B)$ be a pair in $\SL(\Fp)$ which generates $A_5$ in $\PSL(\Fp)$. The coordinates of $\Tr(A, B)$ lie in $\{0, \pm 1, \frac{\pm 1 \pm \sqrt{5}}{2}\}$, and at least two of those coordinates are non-zero since $A_5$ is not a dihedral group. Moreover, Theorem \ref{towers} shows that the exceptional orbits associated with $A_5$ and $A_4$ must be disjoint.

\par We first show that every exceptional orbit associated with $A_5$ contains a triple with a zero coordinate. Assume by contradiction that this is not the case for a given orbit. Then, all coordinates in such orbit lie in $\{\pm 1, \frac{\pm 1 \pm \sqrt{5}}{2}\}$. Let $\phi = \frac{1 + \sqrt{5}}{2}$. By choosing the signs of $A$ and $B$, we can assume that the first two coordinates of $\Tr(A, B)$ lie in $\{1, \phi, \phi^{-1}\}$. The choice of signs only affects the sign of two coordinates of the triples, hence creates a bijective correspondence between the two orbits without affecting the condition on non-zero coordinates.

\par Let $(x, y, z) = \Tr(A, B)$. We now consider each case separately.
\begin{itemize}
    \item[1.] $(x, y) = (1, 1)$. In this case, $z = \frac{\pm 1 \pm \sqrt{5}}{2}$, otherwise this orbit is either affine or it coincides with the exceptional orbit of $A_4$. Then, $(1, 1, z) \mapsto (1, z- 1, z)$ via a Vieta move, therefore $z = \frac{1 \pm \sqrt{5}}{2}$. Then $z(z - 1) = 1$, and again $(1,z-1, z) \mapsto (0, z-1, z)$, contradicting the initial hypothesis on non-zero coordinates.
    \item[2.] $(x, y) = (1, \phi)$. By the previous case and by changing signs if needed, we can assume that no triple in this orbit contains two coordinates equal to $ \pm 1$. Consider the Vieta move $(1, \phi, z) \mapsto (1, \phi, \phi - z)$. Then, since $z, \phi - z \ne \pm 1$, the only solution is $z = \phi$, contradicting the condition on non-zero coordinates.
    \item[3.] $(x, y) = (1, \phi^{-1})$. This case is solved just like the previous one, replacing $\phi$ by $\phi^{-1}$.
    \item[4.] $(x, y) = (\phi, \phi)$. Using the previous cases, we can assume that no triple in this orbit contains a coordinate equal to $\pm 1$. Therefore, $z = \pm \phi^{\pm 1}$. Consider the Vieta move $(\phi, \phi, z) \mapsto (\phi, \phi, \phi^2 - z)$. Then $z = \phi$, and $\phi^2 - z = 1$, which cannot happen given the previous cases.
    \item[5.] $(x, y) = (\phi^{\pm 1}, \phi^{\pm 1})$. By all the previous cases, we can assume that no triple in this orbit contains a coordinate equal to $\pm 1$ or has two coordinates equal to $(\pm \phi, \pm \phi)$. Therefore, after choosing signs, we can assume this orbit has a triple where the first two coordinates are $(\phi^{-1}, \phi^{-1})$. This case is solved just like previous one, replacing $\phi$ by $\phi^{-1}$. This concludes the analysis of cases.
    
\end{itemize}

\par This proves that every exceptional orbit associated with $A_5$ contains a triple with a zero coordinate. Using this, we can find a triple in the exceptional orbit with its last coordinate equal to zero, and by using the Vieta move we can choose the sign of its non-zero coordinates. Hence, we can assume this triple is of the form
\[(\phi, \phi, 0),\text{ } (\phi^{-1}, \phi^{-1}, 0),\text{ } (\phi, \phi^{-1}, 0), \text{ } (\phi, 1, 0), \text{ or } (\phi^{-1}, 1, 0).\]
However, no triple in this orbit has a coordinate equal to $\phi^2$ or $\phi^{-2}$, which rules out the first two cases. We are left with three possible orbits of sizes seventy-two, forty, and forty, generated by the triples $(\phi, \phi^{-1}, 0), \text{ } (\phi, 1, 0)$, and $(\phi^{-1}, 1, 0)$. These triples lie at levels $1,$ $\phi$, and $-\phi^{-1}$, and therefore generate disjoint orbits. It remains to prove that these orbits in fact appear as the image of generating pairs of $A_5$ in $\PSL(\Fp)$.

\par Consider the presentation of $A_5$ given by $\langle x, y \text{ } | \text{ } x^5 = y^2 = (xy)^3 = 1\rangle$. First, \mbox{we prove} that the the orbit containing $(\phi, \phi^{-1}, 0)$ appears as an exceptional orbit associated with $A_5$. Let $z = yx^2y$ be an element of order five. Notice that $(x, z)$ generates $A_5$ as $xz^3x = xyxyx = y$. Moreover, $(xz)^2 = (xyx^2y)^2 = (xyx)(xyxyxy)(yxy) = (xy)^3 = 1$. Therefore, for $(A, B)$ in $\SL(\Fp)$ mapping to $(x, z)$ in $\PSL(\Fp)$, we have by Lemma \ref{order-trace} that $\Tr(A, B)$ belongs to the exceptional orbit containing $(\phi, \phi^{-1}, 0)$.

\par Next, we prove that both orbits containing $(\phi, 1, 0)$ and $(\phi^{-1}, 1, 0)$ appear as exceptional orbits associated with $A_5$. We start by taking a pair $(A, B)$ in $\SL(\Fp)$ mapping to $(x, y)$ in $\PSL(\Fp)$, and by choosing signs in $\SL(\Fp)$ we assume that $\tr(A) = \phi^{\pm 1}$. By Lemma \ref{order-trace}, $\Tr(A, B)$ belongs to either one these orbits. Let $u = yx^2y$ and $v = xyx^{-1}$ be elements in $A_5$ of order five and two. Notice that $u^3v = yxyxyx^{-1} = x^{-2}$, so that $(u^3v)^2 = x$ and $(u, v)$ generates $A_5$. Moreover, we have $uv = yxy(yxyxy)x^{-1} = yxyx^{-2} = (yxyx) x^2 = x^{-1}yx^2$, so that $(uv)^3 = x^{-1}yxyxyx^{2} = 1$. Observe that the pair $(BA^2B, ABA^{-1})$ maps to $(u, v)$ in $\PSL(\Fp)$, and therefore it also corresponds to one of the two exceptional orbits associated with $A_5$ mentioned above. On the other hand, as $B^2 = -I$, we have $\tr(BA^2B) = - \tr(A^2) = - \tr(A)^2 + 2$. Therefore, when $\tr(A) = \phi$ we have $\tr(BA^2B) = - \phi^{-1}$, and when $\tr(A) = \phi^{-1}$ we have $\tr(BA^2B) = \phi$. Thus, $\Tr(A, B)$ and $\Tr(BA^2B, ABA^{-1})$ lie in the two distinct orbits containing $(\phi, 1, 0)$ and $(\phi^{-1}, 1, 0)$.

\par Therefore, when $A_5$ is admissible, there are three exceptional orbits associated with it: one containing $(\phi, \phi^{-1}, 0)$ of size seventy-two, one containing $(\phi, 1, 0)$ of size forty, and one containing $(\phi^{-1}, 1, 0)$ of size forty. This completes our classification of exceptional orbits.

\subsection{Classification} \hfill

\par Using the results of this section, we now give a classification of all exceptional orbits:

\begin{thm}\label{orbits}
A list of all exceptional orbits, their corresponding groups, generators, levels, and sizes is given by the following table:
\begin{center}
\begin{table}[ht]
\caption{Exceptional Orbits}
\begin{tabular}{ c  c  c  c}
 Group  & Generator & Level & Size \\ 
 \hline
 \hline
 \multirow{2}{67pt}{Dihedral ($D_n$)}
 & $(0, 0, 0)$ & $-2$ & $1$ \\
 & $(t, 0, 0)$ & $t^2 - 2$ & $6$ \\
 \hline
 Tetahedral ($A_4$) & $(1, 1, 0)$ & $0$ & $16$\\
 \hline
 Octahedral ($S_4$) & $(\sqrt{2}, 1, 0)$ & $1$ & $36$ \\
 \hline
 \multirow{3}{78pt}{Icosahedral ($A_5$)}
 & $(\frac{1+\sqrt{5}}{2}, \frac{1-\sqrt{5}}{2}, 0)$ & $1$ & $72$ \\
 & $(\frac{1+\sqrt{5}}{2}, 1, 0)$ & $\frac{1+\sqrt{5}}{2}$ & $40$ \\
 & $(\frac{1-\sqrt{5}}{2}, 1, 0)$ & $\frac{1-\sqrt{5}}{2}$ & $40$ \\
 \hline
\end{tabular}
\label{exceptional}
\end{table}
\end{center}
\end{thm}

\par Notice that the orbits associated with the octahedral and icosahedral groups appear if and only if $\sqrt{2} \in \Fp$ and $\sqrt{5} \in \Fp$. Via Table \ref{exceptional}, we can find the union of all exceptional orbits at a given level $k$, which we denote by $\ee_k$.

\par Remarkably, the exceptional orbits found above are exactly the finite orbits for the equations $\M_k$ when viewed over $\mathbb{C}$. This result was proved by Dubrovin and Mazzocco in \cite{dubrovin2000monodromy} and provides an interesting perspective on the previous classification. By Chebotarev's theorem, every finite algebraic orbit of $\M_k$ over $\mathbb{C}$ must appear for infinitely many primes. The classification above shows that the converse is essentially true, in the sense that the only orbits that appear for infinitely many primes are in fact the image of orbits over $\mathbb{C}$.

\par We remark that the only exceptional orbit at the singular level $k = 2$ is associated with $D_1 \cong \Z/2\Z$, which therefore characterizes the only intersection between the classes of affine and exceptional subgroups of $\PSL(\Fp)$. With this information in hands, we proceed to make a precise statement of the problem of strong approximation.

\section{{\it Goal:} Strong Approximation} \label{strong}

\par We have associated a type of subgroup of $\PSL(\Fp)$ to each orbit in $\M_k$, and for $k \ne 2$, Theorem \ref{orbits} shows that all but a very small list of finite orbits correspond to the projective subgroup $\PSL(\Fp)$, so that towers over that orbit only contain generating pairs of $\SL(\Fp)$. As previously noted, the spirit of strong approximation lies in the idea that there is always one large orbit at each level $k \ne 2$ which is associated with generating pairs. Using \mbox{Table \ref{exceptional}}, we give this statement a precise formulation in Section \ref{conjs}.

\par In Section \ref{connecting_towers}, we prove that the Q-classification is equivalent to strong approximation when $p \equiv 3 \pmod{4}$. In Section \ref{divisibility}, we discuss some recent developments towards proving strong approximation as well as divisibility conditions on the size of generating orbits that imply strong approximation.

\subsection{Conjectures}\label{conjs} \hfill

\par Notice that the existence of a large orbit at each level $k \ne 2$ has a natural interpretation in terms of generating pairs of $\SL(\Fp)$. This idea was introduced at the end of Section \ref{Nielsen}, and using the machinery developed so far we can now reinterpret it in more familiar terms. In fact, Theorem \ref{towers} shows that each tower over a generating orbit contains at most two Nielsen orbits.

\par Assuming the hypothesis of strong approximation, there is a single generating orbit of triples at each level $k$, so that there are at most two Nielsen orbits of generating pairs $(A, B)$ for which $\tr([A, B]) = k$. This motivates the Q-classification conjecture mentioned at the end of Section \ref{Nielsen}, proposed by McCullough and Wanderley (c.f. \cite{mccullough2013nielsen}). In group theoretical terms, it states that two generating pairs of $\SL(\Fp)$ are in the same Nielsen orbit if and only if their commutators lie in the same extended conjugacy class of $\SL(\Fp)$.

\par As noted at the end of Section \ref{Nielsen}, Nielsen moves preserve extended conjugacy classes. The Q-classification conjecture states that, in $\SL(\Fp)$, this is a complete invariant of the Nielsen orbits. We recall that for $k \ne \pm 2$, there is a single conjugacy class consisting of all elements of trace $k$ in $\SL(\Fp)$. As $\tr(A) = \tr(A^{-1})$, this conjugacy class is also an extended conjugacy class of $\SL(\Fp)$.

\par For $k = -2$, there are two conjugacy classes different from $\{ - I \}$ containing all non-central elements of trace $-2$. When $p \equiv 1 \pmod{4}$, these classes are closed under inverse, hence yielding two extended conjugacy classes. When $p \equiv 3 \pmod{4}$, these classes are complementary under inverse and form one extended conjugacy class. For $k = 2$, \mbox{Lemma \ref{trace_thm}} shows that there is no generating pair of $\SL(\Fp)$ with commutator trace equal to $2$. Thus, the Q-classification conjecture can be stated as follows:

\begin{conj}[Q-Classification]\label{Q} There is a single orbit of generating pairs at level \mbox{$k \ne - 2$}. Moreover, for $p \equiv 3 \pmod{4}$, there is also a single orbit of generating pairs at level $-2$, and for $p \equiv 1 \pmod{4}$, there are two such orbits at that level.
\end{conj}

\par We refer to the Q-classification at level $k$ as the previous conjecture restricted to level $k$. For $k \ne -2$ or $p \equiv 3 \pmod{4}$, the Q-classification conjecture clearly implies that there is a single generating orbit at level $k$. For $k = -2$ and $p \equiv 1 \pmod{4}$, the \mbox{same conclusion} still holds as each tower at that level contains pairs in both extended conjugacy classes of trace $-2$, which cannot be in the same Nielsen orbit. This last claim follows from the fact that for $X = \text{diag}(1, t) \in \GL(\Fp)$, $t$ non-square in $\Fp$, and $(A, B)$ a pair at level $k = -2$, the pairs $(A, B)$ and $(A^X, B^X)$ have different extended conjugacy classes. Therefore, the Q-classification implies that there is a single generating orbit at each level $k \ne 2$.

\par The idea that there is a large orbit in each non-singular level was familiar to Bourgain, Gamburd, and Sarnak, who made major contributions to the problem (c.f. \cite{bourgain2016markoff}). Their main result shows that at level $k = -2$ and for most primes, there is a single orbit $\cc_{-2}$ referred to as the cage in addition to $\{(0, 0, 0)\}$. Moreover, even when this fails, there are effectively at most $p^{\epsilon}$ triples outside $\cc_{-2}$.

\par In \cite{announcement}, Bourgain, Gamburd, and Sarnak announced that the same result extends to other non-singular levels via similar methods. Their result captures the essence that there is a single generating orbit at each level $k \ne 2$, which we state as follows:

\begin{thm}\label{sarnak}
For $k \ne 2$, the set of solutions for $\M_k$ contains a large orbit $\mathcal{C}_k$ such that \[|\M_k(\Fp) \setminus \mathcal{C}_k| < p^{\epsilon}.\]
\end{thm}

\par The orbit $\cc_k$ is referred to as the cage at level $k$. In Section \ref{divisibility}, we show that $\M_k(\Fp)$ has approximately $p^2$ triples, so that most triples of $\M_k(\Fp)$ lie in $\cc_k$. We remark that this result is effective on $\epsilon$.

\par As strong approximation expects a single generating orbit at each level $k \ne 2$, we can classify these orbits by classifying their complement in $\M_k(\Fp)$. In fact, each non-singular level $\M_k(\Fp)$ is the union of generating orbits and exceptional orbits. Keeping the previous notation, $\ee_k$ denotes the union of all exceptional orbits at level $k$, as given by Theorem \ref{main}. Then, strong approximation can be analytically formulated as follows:

\begin{conj}[Strong Approximation]
For each level $k \ne 2$, there is a single orbit $\mathcal{C}_k$ such that \[\M_k(\Fp) = \mathcal{C}_k \sqcup \mathcal{E}_k.\]
\end{conj}

\par In other words, the cage $\cc_k$ is expected to be $\M_k(\Fp) \setminus \mathcal{E}_k$. As for the Q-classification, we refer to strong approximation at level $k$ as the previous conjecture restricted to level $k$.

\par Our next goal is to study the relationship between the Q-classification and strong approximation. We have already noted that the Q-classification implies strong approximation. On the other hand, Theorem \ref{towers} tells us that the converse might also be possible, as connectedness on the base $\mathcal{C}_k = \M_k(\Fp) \setminus \mathcal{E}_k$ implies that there are at most two Nielsen orbits of generating pairs at that level. In order to show that the Nielsen orbit is unique for most cases, we now attempt to connect towers via Nielsen moves.
\newpage

\subsection{Connecting towers}\label{connecting_towers} \hfill

\par We now focus on the orbits of each tower $\E(x, y, z)$. By Theorem \ref{towers} and \mbox{Corollary \ref{haha}}, $\E(x, y, z)$ contains at most two orbits. The goal of this section is to prove that all towers over the cage $\mathcal{C}_k$ contain a single orbit when $p \equiv 3 \pmod{4}$, thus reducing the Q-classification conjecture to strong approximation. In order to do that, we will use special triples at level $k$ which naturally induce a connection in the tower $\E(x, y, z)$. The existence of such triples is given by following result:

\begin{prop}\label{weyl}
Let $k \in \Fp$ be such that $\left(\frac{2-k}{p}\right) = 1$. Then, the equation \[\M_k: x^2 + y^2 + z^2 - xyz - 2 = k\] has at least $\frac{p}{16}$ solutions such that $x$ and $y$ are hyperbolic and $xy - z= z$.\end{prop}

\par Such triples $(x, y, z)$ at level $k$ with $x$ and $y$ hyperbolic and $xy - z = z$ will induce the desired connection within the tower $\E(x, y, z)$. This argument is presented in detail later in the section. For now, we focus on proving Proposition \ref{weyl}. As it turns out, this is equivalent to counting the number of points on certain curves over $\Fp$. This is done via two technical results. The first one is much simpler:

\begin{lemma}\label{square}
Let $c \in \Fp^{\times}$. Then, there are $p-1$ pairs $(x, y)$ such that \[x^2 + c = y^2.\]
\end{lemma}

\begin{proof}
Let $(z, w) = (x + y, x - y)$. Then, the equation is equivalent to $zw = -c$, which has $p-1$ solutions $(z, w)$. Hence the original equation has $p-1$ solutions $(x, y)$.
\end{proof}

\begin{cor}\label{c1}
There are at least $\frac{p-1}{2}$ values of $x$ for which $\left(\frac{x^2 + c}{p}\right) = 1$.
\end{cor}

\par This first result is used to deal with corner cases in the main argument. We now present the main result, which is a natural generalization of Lemma \ref{square}.

\begin{prop}\label{squares}
Let $a, b \in \Fp^{\times}$ be distinct. Let $N$ be the number of triples $(x, y, z)$ \mbox{such that}
\[\begin{cases}x^2 + a = y^2 \\ x^2 + b = z^2.\end{cases}\]
Then, $|N - (p-3)| \le 2 \sqrt{p}$.
\end{prop}

\begin{proof}
Define the function $\e(t) = \exp(\frac{2\pi i t}{p})$ for $t$ in $\Fp$. Consider the double sum
\[ \Lambda_{a, b} =  \sum_{\alpha, \beta} \sum_{x, y, z} \e(\alpha(x^2+a-y^2)) \e(\beta(x^2+b-z^2)),\] where the indices range over $\Fp$ unless otherwise stated. Notice that if $(x, y, z)$ is a solution for the system, the sum over $\alpha$ and $\beta$ contributes $p^2$ to the total sum, otherwise it contributes zero, as $\sum_{\gamma} \e(\gamma t) = 0$ for $t \ne 0$. Therefore $\Lambda_{a, b} = p^2N$.

\par We now compute $\Lambda_{a, b}$ by exchanging the order of the sums. Notice that
\[ \Lambda_{a, b} = \sum_{\alpha, \beta, x, y} \e(\alpha(x^2 + a - y^2)) \e(\beta(x^2 + b)) \sum_z \e(-\beta z^2).\]
Denote the quadratic Gauss sum by $\theta(c) = \sum_{w} \e(cw^2)$. We further reduce our \mbox{expression to} 
\begin{equation*} \begin{split}
\Lambda_{a, b} & = \sum_{\alpha, \beta, x, y} \e(\alpha(x^2 + a - y^2)) \e(\beta(x^2 + b)) \theta(-\beta) \\
& = \sum_{\alpha, \beta, x} \e(\alpha(x^2 + a))\e(\beta(x^2 + b)) \theta(-\beta) \sum_{y} \e(-\alpha y^2) \\
& = \sum_{\alpha, \beta, x} \e(\alpha(x^2 + a))\e(\beta(x^2 + b)) \theta(-\beta) \theta(-\alpha) \\
& = \sum_{\alpha, \beta} \e(\alpha a + \beta b) \theta(-\beta)\theta(- \alpha) \sum_x \e((\alpha + \beta)x^2) \\
& = \sum_{\alpha, \beta} \e(\alpha a + \beta b) \theta(-\beta) \theta(-\alpha) \theta(\alpha + \beta).
\end{split}
\end{equation*}

\par Notice that $\theta$ only takes three values since $\theta(c_1) = \theta(c_2)$ whenever $(\frac{c_1}{p}) = (\frac{c_2}{p})$. Hence, we can define $\theta_0, \theta_{\pm 1}$ such that $\theta(c) = \theta_{(\frac{c}{p})}$ for all $c \in \Fp$. Moreover, let $\gamma_c = \theta_c/\sqrt{p}$. We now go back to the expression of $\Lambda_{a, b}$. As $\theta(c) = \theta_{(\frac{c}{p})} = \sqrt{p} \gamma_{(\frac{c}{p})}$, we obtain
\begin{equation*}
\begin{split}
\Lambda_{a, b} & = p^{3/2}\sum_{\alpha, \beta} \e(\alpha a + \beta b) \gamma_{(\frac{-\beta}{p})} \gamma_{(\frac{-\alpha}{p})} \gamma_{(\frac{\alpha + \beta}{p})} \\
& = p^{3/2} \left( p^{3/2} + \sqrt{p} \sum_{\beta \ne 0}  \e(\beta b) \gamma_{(\frac{-\beta}{p})} \gamma_{(\frac{\beta}{p})} + \sqrt{p} \sum_{\alpha \ne 0}  \e(\alpha a) \gamma_{(\frac{-\alpha}{p})} \gamma_{(\frac{\alpha}{p})}\right) \\
& + p^{3/2} \left( \sqrt{p} \sum_{\alpha \ne 0} \e(\alpha(a - b)) \gamma_{(\frac{\alpha}{p})} \gamma_{(\frac{-\alpha}{p})} + \sum_{\substack{\alpha, \beta,\\ \alpha + \beta \ne 0}} \e(\alpha a + \beta b) \gamma_{(\frac{-\beta}{p})} \gamma_{(\frac{-\alpha}{p})} \gamma_{(\frac{\alpha + \beta}{p})}\right) \\
& = p^3 + p^2 \left( \sum_{\beta \ne 0} \e(\beta b) \left(\frac{\beta^2}{p}\right)  + \sum_{\alpha \ne 0} \e(\alpha a) \left(\frac{\alpha^2}{p}\right) + \sum_{\alpha \ne 0} \e(\alpha(a-b)) \left(\frac{\alpha^2}{p}\right) \right) \\
& + p^{3/2} \left(\sum_{\substack{\alpha, \beta,\\ \alpha + \beta \ne 0}} \e(\alpha a + \beta b) \gamma_{(\frac{-\beta}{p})} \gamma_{(\frac{-\alpha}{p})} \gamma_{(\frac{\alpha + \beta}{p})}\right) \\
& = p^3 - 3p^2 + p^{3/2} \left(\sum_{\substack{\alpha, \beta,\\ \alpha + \beta \ne 0}} \e(\alpha a + \beta b) \gamma_{(\frac{-\beta}{p})} \gamma_{(\frac{-\alpha}{p})} \gamma_{(\frac{\alpha + \beta}{p})}\right).
\end{split}
\end{equation*}

\par It is clear that $\gamma_0 = \sqrt{p}$. Moreover, by a classical result of Gauss (c.f. \cite{berndt1998gauss}), we know $\gamma_j = \epsilon_p j$ for $j = \pm 1$, where $\epsilon_p = 1$ if $p \equiv 1 \pmod{4}$ and $\epsilon_p = i$ if $p \equiv 3 \pmod{4}$. Thus, in particular, $\gamma_{r} \gamma_{s} =  rs (\frac{-1}{p})$ for $r, s = \pm 1$. Now, recall that $\Lambda_{a, b} = p^2N$. By dividing the expression above by $p^2$, we obtain
\begin{equation*}
\begin{split}
N+3 - p & = \frac{1}{\sqrt{p}} \sum_{\substack{\alpha, \beta,\\ \alpha + \beta \ne 0}} \e(\alpha a + \beta b) \gamma_{(\frac{-\beta}{p})} \gamma_{(\frac{-\alpha}{p})} \gamma_{(\frac{\alpha + \beta}{p})} \\
 & = \frac{1}{\sqrt{p}} \sum_{\substack{\alpha, \beta,\\ \alpha + \beta \ne 0}} \e(\alpha a + \beta b) \left(\frac{ - \alpha \beta}{p} \right) \gamma_{(\frac{\alpha + \beta}{p})} \\
 & =  \frac{\epsilon_p}{\sqrt{p}} \sum_{\substack{\alpha, \beta,\\ \alpha + \beta \ne 0}} \e(\alpha a + \beta b) \left(\frac{-\alpha \beta(\alpha + \beta)}{p} \right)  \\
& = \frac{\epsilon_p(-1)^{\frac{p-1}{2}}}{\sqrt{p}} \sum_{\alpha, \beta} \e(\alpha a + \beta b) \left(\frac{ \alpha \beta(\alpha + \beta)}{p} \right).
\end{split}
\end{equation*}
\par Notice that whenever $\alpha, \beta,$ or $\alpha + \beta$ equals zero, its corresponding term is zero in the summation. Hence, we can write $\beta = \alpha \eta$ so that

\begin{equation*}
\begin{split}
N+3 - p & = \frac{\epsilon_p(-1)^{\frac{p-1}{2}}}{\sqrt{p}} \sum_{\alpha, \eta} \e(\alpha( a + \eta b)) \left(\frac{ \alpha^3 \eta(1 + \eta)}{p} \right) \\
& = \frac{\epsilon_p(-1)^{\frac{p-1}{2}}}{\sqrt{p}} \sum_{\eta} \left(\frac{\eta(1 + \eta)}{p} \right) \sum_{\alpha} \e(\alpha(a+\eta b)) \left(\frac{ \alpha}{p} \right) \\
& =  \frac{\epsilon_p(-1)^{\frac{p-1}{2}}}{\sqrt{p}} \left( \left(\frac{a(a-b)}{p}\right) \sum_{\alpha}  \left(\frac{ \alpha}{p} \right) + \sum_{\eta \ne -a/b} \left(\frac{\eta(1 + \eta)(a+\eta b)}{p} \right) \sum_{\alpha'} \e(\alpha') \left(\frac{ \alpha'}{p} \right)  \right)\\
& = \epsilon_p^2(-1)^{\frac{p-1}{2}} \sum_{\eta} \left(\frac{\eta(1 + \eta)(a+\eta b)}{p} \right) \\
& = \sum_{\eta} \left(\frac{\eta(1 + \eta)(a+\eta b)}{p}\right),
\end{split}
\end{equation*}
using that $\sum_{t} (\frac{ t}{p})  = 0$ and that the value of the Gauss sum $\sum_{t} \e(t) (\frac{ t}{p} )$ is $\epsilon_p \sqrt{p}$. We remark that $N$ is in fact a symmetric function of $a$ and $b$, as the change of variables $\eta \mapsto \eta a/b$ yields this symmetry above. We now compute the last expression.

\par Consider the elliptic curve over $\Fp$ given by $\zeta^2 = \eta(\eta+1)(\eta+c)$, where $c = a/b$. This curve has a single point at infinity and $M_c$ other points, therefore the Hasse bound gives us that $|M_c - p| \le 2\sqrt{p}$. There is, however, an explicit way of counting $M_c$ which relates to our previous expression. In fact,
\[M_c = \sum_{\eta} 1 + \left(\frac{\eta(1 + \eta)(c+\eta)}{p} \right),\]
as whether or not $\eta(1 + \eta)(c+\eta)$ is a square dictates how many points $(\zeta, \eta)$ the curve has. Therefore, we have
\[\sum_{\eta} \left(\frac{\eta(1 + \eta)(c+\eta )}{p} \right) = M_c - p.\]
Finally, we obtain that $N + 3 - p = \left(\frac{b}{p}\right)(M_c - p)$, and in particular $|N - (p-3)| \le 2\sqrt{p}.$
\end{proof}

\begin{cor}\label{c2}
There are at least $\frac{p-3 - 2\sqrt{p}}{4}$ values of $x$ for which $\left(\frac{x^2 + a}{p}\right) = \left(\frac{x^2 + b}{p}\right) = 1$.
\end{cor}

With these results in hand, we are now ready to prove Proposition \ref{weyl}.

\begin{proof}[Proof of Proposition \ref{weyl}]
Let $u, v \in \Fp^{\times} \setminus \{\pm 1\}$, and suppose that the triple $(x, y, z) = (u+u^{-1}, v + v^{-1}, \frac{1}{2}(u+u^{-1})(v+v^{-1}))$ is a solution for $\M_k$. As $z = \frac{xy}{2}$, this is equivalent to $x^2 + y^2 - \frac{1}{4} x^2y^2 = k+2 \iff (x^2 - 4)(y^2 - 4) = 4(2-k)$. Writing $2 - k = \frac{t^2}{4}$, $t \ne 0$, we get
\[(u-u^{-1})(v-v^{-1}) = \pm t.\]
\par We now estimate the number of pairs $(u, v)$ satisfying $(u-u^{-1})(v-v^{-1}) = t$ for a given $t \ne 0$. As $t \ne -t$ in the problem, the two corresponding sets of solutions for $\M_k$ will be disjoint. This is equivalent to solving $u - u^{-1} = w, v - v^{-1} = t/w$, with $w \ne 0$ variable. Following Lemma \ref{geometry}, this condition is implied by $w \ne 0$ and \[\left(\frac{w^2 + 4}{p}\right) = \left(\frac{w^2 + t^2/4}{p} \right) = 1.\]

\par By Corollaries \ref{c1} and \ref{c2}, there are at least $\frac{p-7-2\sqrt{p}}{4} > \frac{p}{8}$ values of $w$ for which this condition is true. Each value of $w$ yields at least one solution $(u, v)$, and as we can consider the previous system for $\pm t$, we obtain at least $\frac{p}{4}$ pairs. Finally, notice that each \mbox{desired triple} $(x, y, z)$ is generated by at most four pairs $(u, v)$. Therefore, there are at least $\frac{p}{16}$ triples for which the initial condition holds. \end{proof}

\par We now have the tools to prove the main result of this section. However, before doing so, we present a related result of McCullough and Wanderley (c.f. \cite{mccullough2013nielsen}) which shows that when $2 - k$ is not a square, a generating orbit at level $k$ is the image of a unique Nielsen orbit of pairs under the trace map. This is translated into the following statement:

\begin{thm}\label{mcwan}
Let $k \in \Fp$ be such that $2-k$ is not a square. Then, the tower $\E(x, y, z)$ contains a single Nielsen orbit for every $(x, y, z)$ in a generating orbit at level $k$. 
\end{thm}

\par This result follows from the fact that the pairs $(A, B)$ and $(A^{-1}, B^{-1})$ are always in the same tower and in the same orbit, but are conjugate in $\SL(\Fp)$ only if $2-k$ is a square. We will extend this result in the case $p \equiv 3 \pmod{4}$, proving that all towers over the cage $\mathcal{C}_k$ contain a single orbit provided that $2-k$ is a non-zero square. In fact, we are not able to prove directly that all towers are connected, but we show that at least one tower in the cage $\mathcal{C}_k$ is. By transitivity, $\mathcal{C}_k$ is the image of a unique orbit under the trace map.

\begin{thm}\label{equivalence}
Let $p \equiv 3 \pmod{4}$ and $k \in \Fp$ be such that $2 - k$ is a non-zero square. Then, the tower $\E(x, y, z)$ contains a single Nielsen orbit for every $(x, y, z)$ in the cage $\mathcal{C}_k$.
\end{thm}

\begin{proof}
Let $(x, y, z)$ be a solution given by Proposition \ref{weyl} which lies inside the cage $\mathcal{C}_k$. Such solution exists, as by Theorem \ref{sarnak} there are at most $p^{\epsilon}$ solutions of $\M_k$ outside $\mathcal{C}_k$. Write \[x = u+u^{-1},\text{ } y = v + v^{-1},\text{ } z = \frac{1}{2} (u+u^{-1})(v+v^{-1}).\]
\par As $x$ is hyperbolic, there is a pair $(A, B)$ in $\E(x, y, z)$ such that
\[A = \begin{pmatrix} u & 0 \\ 0 & u^{-1} \end{pmatrix}, \text{ }B = \begin{pmatrix} a & b \\ c & d \end{pmatrix}.\]
\par Notice that $\tr(A^{-1}B) = \tr(A)\tr(B) - \tr(AB) = \tr(AB)$, since $xy -z = z$. Then, the pair $(A^{-1}, B)$ is also in $\E(x, y, z)$. Notice, however, that this new pair was not obtained by conjugation in $\SL(\Fp)$ of $(A, B)$, at least not explicitly. In fact, we will show that $(A^{-1}, B)$ can only be conjugate in $\SL(\Fp)$ to $(A, B)$ under certain conditions. Assume that there is $X \in \SL(\Fp)$ such that $A^X = A^{-1}$, $B^X = B$. The first equation yields \[X = \begin{pmatrix} 0 & -s \\ s^{-1} & 0 \end{pmatrix}\]
for some $s \in \Fp$. Using $B^X = B$, we obtain $a = d, -b = cs^2$, and in particular, $-bc = (cs)^2$. As the pair $(A, B)$ is in normal form, Lemma \ref{normal} gives us that \[2 - k = bc(x^2 - 4).\]
Thus $k - 2 = (cs(u-u^{-1}))^2$ is a square.
\par Therefore $(A, B)$ and $(A^{-1}, B)$ can only be conjugate in $\SL(\Fp)$ when $k-2$ is a square, which is not the case. Hence, it follows that the Nielsen move mapping $(A, B)$ to $(A^{-1}, B)$ connects the two $\SL(\Fp)$ conjugacy classes of the tower $\E(x, y, z)$, so that $\E(x, y, z)$ consists of a single Nielsen orbit. By transitivity in $\mathcal{C}_k$, it follows that all pairs $(A, B)$ above $\mathcal{C}_k$ are in the same Nielsen orbit.
\end{proof}

As previously noted, the Q-conjecture implies strong approximation. On the \mbox{other hand}, the past two results show the equivalence between these problems in the case $p \equiv 3 \pmod{4}$. In fact, it suffices to see that the towers over the generating orbit $\mathcal{C}_k$ are connected: when $2-k$ is not a square, this follows from Theorem \ref{mcwan}, and when $2-k$ is a non-zero square, this follows from Theorem \ref{equivalence}. Therefore,

\begin{thm}
The Q-classification is equivalent to strong approximation when $p \equiv 3 \pmod{4}$. Moreover, the two conjectures are equivalent at levels $k$ whenever $2-k$ is not a square or $k = -2$.
\end{thm}

\par The equivalence between Q-classification and strong approximation in the case $k = -2$ and $p \equiv 1 \pmod{4}$ comes from the fact that there are two Nielsen orbits at that level, as remarked after Conjecture \ref{Q}.

\par A few comments on the case $p \equiv 1 \pmod{4}$ are relevant at this point. The difficulty in extending this result using the previous method is that the conditions that $2-k$ is a square, and $k - 2$ is a square, are equivalent. Moreover, we cannot expect to prove the existence of connected towers at level $k = -2$, as they in fact contain two Nielsen orbits. Therefore, conditions on $k+2$ seem to be more appropriate in this case, although this quantity does not seem to arise as naturally as $2-k$.

\subsection{Divisibility Conditions}\label{divisibility} \hfill

\par A significant step towards proving strong approximation was recently taken by Chen, who confirmed the result in the case $k = -2$ (c.f. \cite{chen2020strong}). Chen's work shows that $p$ divides the size of the cage $\cc_{-2}$, which combined with Theorem \ref{sarnak} and the fact that $|\M_{-2}(\Fp)| \equiv 1 \pmod{p}$, proves that $\cc_{-2} = \M_{-2}(\Fp) \setminus \ee_{-2}$, where $\ee_{-2} = \{(0, 0, 0)\}$. Moreover, Chen's results also yield weaker divisibility conditions for other levels $k \ne -2$, which imply strong approximation under certain conditions. These results are described below:

\begin{thm}\label{chen}
Let $\Orb$ be an orbit at level $k \ne 2$ of triples with at least two non-zero coordinates and let $h_k$ be the order of any non-central element of $\SL(\Fp)$ with trace $k$. Then,
\begin{itemize}
    \item[(i)] for $\left(\frac{k^2 - 4}{p}\right) = 0$, we have $|\Orb| \equiv 0 \mod{p}$.
    \item[(ii)] for $\left(\frac{k^2 - 4}{p}\right) = 1$, we have $2 |\Orb| \equiv 0 \mod{\frac{h_k}{\gcd\left(h_k, \frac{4(p-1)}{h_k}\right)}}$.
    \item[(iii)] for $\left(\frac{k^2 - 4}{p}\right) = -1$, we have $2 |\Orb| \equiv 0 \mod{\frac{h_k}{\gcd\left(h_k, \frac{4(p+1)}{h_k}\right)}}$.
\end{itemize}
\end{thm}

\par Notice that this produces divisibility conditions even for exceptional orbits associated with $A_4$, $S_4$, and $A_5$. Remarkably, when $h_k$ is sufficiently larger than $\gcd(h_k, 4(p\pm 1)/h_k)$, this result combined with Theorem \ref{sarnak} imply strong approximation at level $k$.

\par We now investigate the expected size of each cage $\cc_k$, with the goal of conjecturing a stronger divisibility condition in the general case. The first step is to count the number of \mbox{solutions of $\M_k$}:

\begin{lemma}\label{count}
The number of solutions for the equation $\M_k: x^2 + y^2 + z^2 - xyz - 2 = k$ is  \[|\M_k(\Fp)| = p^2 + \left(3 + \left(\frac{k+2}{p}\right)\right)\left(\frac{k-2}{p}\right) p + 1.\]
\end{lemma}

\begin{proof}
We keep the notation from the previous section for the quadratic Gauss sum $\theta(c) = \sum_{w} \e(cw^2)$ and $\gamma_j = \theta_j / \sqrt{p}$ for $j = 0, \pm 1$. Notice that
\begin{equation*}
\begin{split}
    p \text{ }|\M_k(\Fp)| & = \sum_{\alpha} \sum_{x, y, z} \e(\alpha(x^2 +y^2+z^2 - xyz - 2 - k)) \\
    & = \sum_{\alpha, x, y} \e(\alpha(x^2 + y^2 - 2 -k)) \sum_{z} \e(\alpha(z^2 - xyz)) \\
    & = \sum_{\alpha, x, y} \e(\alpha(x^2 + y^2 - 2 - k - x^2y^2/4)) \sum_{z'} \e(\alpha z'^2) \\
    & = \sum_{\alpha, x, y}\e(\alpha(x^2 + y^2 - 2 - k - x^2y^2/4)) \theta(\alpha) \\
    & = \sum_{\alpha, x} \e(\alpha(x^2 - 2 -k)) \theta(\alpha) \sum_{y} \e(\alpha (1 - x^2/4) y^2)\\
    & = \sum_{\alpha, x} \e(\alpha(x^2 - 2 -k)) \theta(\alpha) \theta(\alpha(1-x^2/4)).
\end{split}
\end{equation*}
\par Recall that $\theta(c) = \theta_{(\frac{c}{p})} = \sqrt{p} \gamma_{(\frac{c}{p})}$, $\gamma_0 = \sqrt{p}$, and moreover $\gamma_j = \epsilon_p j$ for $j = \pm 1$. Hence,

\begin{equation*}
\begin{split}
    |\M_k(\Fp)| & = \sum_{\alpha, x} \e(\alpha(x^2 - 2 -k)) \gamma_{(\frac{\alpha}{p})} \gamma_{(\frac{\alpha(1-x^2/4)}{p})} \\
    & = p \sum_{x} \e(0) + 2 \sqrt{p} \sum_{\alpha \ne 0} \e(\alpha(2-k)) \gamma_{(\frac{\alpha}{p})} \\
    & + \sum_{\alpha \ne 0, x \ne \pm 2} \e(\alpha(x^2 - 2 -k)) \gamma_{(\frac{\alpha}{p})} \gamma_{(\frac{\alpha(1-x^2/4)}{p})} \\
    & = p^2 + 2\epsilon_p\sqrt{p} \sum_{\alpha \ne 0} \e(\alpha(2-k))\left(\frac{\alpha}{p}\right) \\
    & + \left(\frac{-1}{p}\right) \sum_{\alpha \ne 0, x \ne \pm 2} \e(\alpha(x^2 - 2 -k)) \left(\frac{\alpha^2(4 - x^2)}{p}\right) \\
    & = p^2 + 2\epsilon_p \sqrt{p} \left(\frac{2-k}{p}\right) \sum_{\alpha'} \e(\alpha')\left(\frac{\alpha'}{p}\right) \\
    & + \left(\frac{-1}{p}\right) \sum_{\alpha \ne 0, x} \e(\alpha(x^2 - 2 - k)) \left(\frac{4-x^2}{p}\right).
\end{split}
\end{equation*}
\par Using that the value of the Gauss sum $\sum_{t} \e(t) (\frac{ t}{p} )$ is $\epsilon_p \sqrt{p}$, we obtain
\begin{equation*}
\begin{split}
    |\M_k(\Fp)| & = p^2 + 2p \left(\frac{k-2}{p}\right) + \left(\frac{-1}{p}\right) \sum_x \left(\frac{4-x^2}{p}\right) \sum_{\alpha \ne 0} \e(\alpha(x^2 - 2 - k)) \\
    & = p^2 + 2p\left(\frac{k-2}{p}\right) - \left(\frac{-1}{p}\right) \left( \sum_{x^2 \ne k + 2} \left(\frac{4-x^2}{p}\right) -(p-1)\sum_{x^2 = k + 2}\left(\frac{2-k}{p}\right)\right) \\
    & =  p^2 + 2p\left(\frac{k-2}{p}\right) - \sum_{x^2 \ne k+2}\left(\frac{x^2 - 4}{p}\right) + (p-1)\left(\frac{k-2}{p}\right) \left( \left(\frac{k+2}{p}\right) + 1 \right).
\end{split}
\end{equation*}

\par From Proposition \ref{square}, we know that $\sum_w (\frac{w^2+c}{p}) = -1$ when $c \ne 0$. Therefore,

\begin{equation*}
\begin{split}
    |\M_k(\Fp)| & = p^2 + \left(3 + \left(\frac{k+2}{p}\right)\right)\left(\frac{k-2}{p}\right)p + 1.
\end{split}
\end{equation*}
The result follows. \end{proof}

We now subtract the union of all exceptional orbits $\ee_k$ at level $k$ from the set of solutions $\M_k(\Fp)$. This yields the expected size of each cage $\cc_k$, and from there we derive a divisibility conjecture. This will be done for levels that contain no orbits associated with $A_4$, $S_4$, \mbox{or $A_5$}, which we refer to as generic.

\par First, assume that the generic level $k$ contains no exceptional orbits. In particular, $k+2$ \mbox{is not} a square. Then, the expected number of elements in the cage $\cc_k$ is \[p^2 + 2\left(\frac{k-2}{p}\right)p + 1.\] Hence, when $k - 2$ is a square we expect to have $(p+1)^2$ elements, \mbox{and when} $k-2$ is not a square we expect to have $(p-1)^2$.

\par Next, assume that the generic level $k$ contains exceptional orbits associated with the dihedral group and $k \ne -2$. This implies that $|\ee_k| = 6$ and that $k+2$ is a square. Then, the expected number of elements in the cage $\cc_k$ is \[p^2 + 4\left(\frac{k-2}{p}\right)p - 5.\] Hence, when $k-2$ is a square we expect to have $(p-1)(p+5)$ elements, and when $k - 2$ is not a square we expect to have $(p+1)(p-5)$.

\par Lastly, for $k = -2$, there will be $p(p\pm 3)$ elements in the cage. In all the three cases, there is a divisibility pattern that appears. This is stated in the following conjecture:

\begin{conj}
Let $\Orb$ be an orbit at a generic level $k$. Suppose that $\Orb$ is not exceptional, that is, it contains a triple with at least two non-zero coordinates. Then, $|\Orb|$ is divisible by \[p - \left(\frac{k^2-4}{p}\right).\]
\end{conj}

\par For the non-generic levels $k = 0, 1, \frac{1 \pm \sqrt{5}}{2}$, the divisibility above fails. This conjecture, combined with Theorem \ref{sarnak}, would suffice to prove strong approximation for all generic levels.

\bibliographystyle{unsrt}
\bibliography{main}{}

\end{document}